\documentclass[10pt,a4paper]{amsart}
\usepackage[english]{babel}
\usepackage{amsfonts}
\usepackage{amsthm}
\usepackage{amsmath}
\usepackage{mathrsfs}
\usepackage{anysize}

\title{Parametrizations of Ideals in $K[x,y]$ and $K[x,y,z]$}
\author{ALEXANDRU CONSTANTINESCU} 
\address{Dipartimento di Matematica, Universit\`a di Genova,
Via Dodecaneso 35, 16146 Genova, Italy}
\email{a.constantinescu@unibas.ch}

\usepackage[english]{babel}
\usepackage{amsfonts}
\usepackage{amsthm}
\usepackage{amsmath}

\theoremstyle{plain}
\newtheorem{theorem}{Theorem}[section]
\newtheorem*{theorem*}{Theorem}
\newtheorem{lemma}[theorem]{Lemma}
\newtheorem{corollary}[theorem]{Corollary}
\newtheorem{proposition}[theorem]{Proposition}
\theoremstyle{definition}
\newtheorem{definition}[theorem]{Definition}

\newtheorem{remark}[theorem]{Remark}

\numberwithin{equation}{section}

\def\HF{Hilbert function}
\def\GB{Gr$\ddot{\textrm{o}}$bner basis}
\def\GBs{Gr$\ddot{\textrm{o}}$bner bases}

\def\Gin{\textup{Gin}}
\def\Lex{\textup{Lex}}
\def\rank{\textup{rank}}

\def\dim{\mathop{\rm dim}\nolimits}

\def\ini{\operatorname{\rm in}}
\def\deg{\operatorname{\rm deg}\nolimits}
\def\det{\operatorname{\rm det}\nolimits}
\def\Min{\mathop{\rm Min}\nolimits}
\def\Max{\mathop{\rm Max}\nolimits}

\def\Hilb{\mathop{\rm Hilb}\nolimits}
\def\PR{\ensuremath{\mathbb{P}}}
\def\AF{\ensuremath{\mathbb{A}}}
\def\GG{\ensuremath{\mathbb{G}}}

\def\a{\ensuremath{\alpha}}
\def\b{\ensuremath{\beta}}
\def\g{\ensuremath{\gamma}}

\def\p{\ensuremath{\phantom{-}}}
\def\pp{\ensuremath{\phantom{^2}}}

\def\k#1{\ensuremath{K[x_{1} , \ldots , x_{#1}]}}

\def\I#1{\ensuremath{V(#1)}}
\def\V#1{\ensuremath{V_{hom}(#1)}}
\def\AA#1{\ensuremath{\mathscr{A}_{#1}}}

\begin{document}
\maketitle
\begin{abstract}
We parametrize the affine space of Artinian affine ideals of $K[x,y]$ which have a given initial ideal with respect to the degree reverse lexicographic term order. The fact that the term order is degree compatible allows us to extend the parametrization to the projective case, namely  zero-dimensional subschemes of $\mathbb{P}^2$ with some extra assumption and to determine the Betti strata of these Gr\"obner cells. This allows us to prove a formula due to A. Iarrobino for the codimension of the Betti strata of codimension two punctual schemes in $\mathbb{P}^2$.
\end{abstract}
\section{Introduction}
Let  $K$ be a filed of any characteristic and $\k{n}$ the polynomial ring in $n$ variables. For a polynomial $f \in \k{n}$  and any term order $\tau$ we denote by $\ini_{\tau}(f)$ the initial term of $f$ with respect to the term order $\tau$. If $I\subset\k{n}$ is an ideal, we denote by $\ini_{\tau}(I)$ the initial ideal of $I$ with respect to $\tau$, that is the monomial ideal generated by $\ini_{\tau}(f)$, for all $f \in I \setminus\{0\}$.\\
 Given  a monomial ideal $I_0 \subset \k{n}$, the set 
$$V_{hom}(I_0) =   \{ I \subset \k{n} ~: ~I \textup{~homogeneous, with~}  \ini_{\tau}(I) = I_0 \}$$ has a natural structure of affine variety, in the sense that an ideal $I \in \V{I_0}$ can be considered as a point in the affine space $\mathbb{A}^N$. The coordinates are  given by the coefficients of the non-leading terms in the reduced \GB~ of the ideal $I$. If  $d=\dim_K(\k{n}/I_0) < \infty$, also the set  in which we consider all ideals (homogeneous or not), 
$$\I{I_0} := \{ I \subset \k{n} ~:~ \ini_{\tau}(I) = I_0 \}$$
 has a structure of affine variety.  

It is important to note that $\I{I_0}$ (respectively $V_{hom}(I_0)$) coincides with the points of the Hilbert scheme $\Hilb^d(\AF^n)$  (respectively $\Hilb^H(\PR^{n-1})$) that degenerate to $I_0$ under a suitable $K^*$-action associated to a weight vector representing the term order on monomials of degree $\le d+1$. Here $H$ is the Hilbert series of $\k{n}/I_0$. 
 By analogy with the Schubert cells for Grassmannians, $\I{I_0}$ and $V_{hom}(I_0)$ are called  Gr\"obner cells.
These varieties play an important role in the study of various types of Hilbert schemes and also in the problem of deforming nonradical to radical ideals, see \cite{brian, concaRV, ES1, ES2, gotsch1, gotsch2, iarro1, iarro, R}.
For ideals $I_0$ of $K[x,y]$ it is known by results of J. Brian\c con \cite{brian} and A. Iarrobino \cite{iarro1}  that  $V(I_0)$  is an affine space. This fact is also a consequence of general results of A. Bialynicki-Birula \cite{bial1, bial2}.

The main goal of this paper is to obtain a parametrization in the affine case in two variables, that may be extended to the projective case in three variables. The only term orders that allow us to do so are degree-compatible term orders. It is not hard to see that the for $K[x,y]$, if one fixes $x>y$, there is only one such term order. Even if it is the same term order as the  degree lexicographic one,  we will call it  the degree reversed-lexicographic (DRL) in order to emphasize that in the extension to the projective setting we will  use the DRL term order induced by $x>y>z$.

Our main result (Theorem \ref{main}) is the parametrization of  the affine variety $\I{I_0}$,  when $I_0$ is a lex-segment  ideal of $R=K[x,y]$,  $\tau$ is the degree reverse-lexicographic  term order induced by $x>y$ and  $\dim_{k}(R/I_{0}) < \infty$.
 The parametrization from Theorem \ref{main} will be then extended to $\V{I_0K[x,y,z]}$ (Theorem \ref{main2}), which will  be shown to be dense in $\Hilb^H(\PR^{n-1})$ when the characteristic is 0 or large enough. The restriction to the lex-segment case in Theorem \ref{main} requires the restriction of the characteristic in order to obtain that $\V{I_0K[x,y,z]}$ is a Zariski open subset of  $\Hilb^H(\PR^{n-1})$.

In \cite{concavalla} A. Conca and G. Valla parametrize in a similar way the variety $V(I_0)$, with respect to the lexicographic term order. As we already mentioned, the fact that the lexicographic term order is not degree-compatible does not allow a lifting of the parametrization to the projective case. Thus the  advantage obtained by Theorem \ref{main} is that it allows us to extend the parametrization to homogeneous ideals of the polynomial ring in three variables, with some extra assumption. 

 This explicit description of the affine space structure of $\I{I_0}$ is obtained by associating to each  ideal a canonical Hilbert-Burch matrix. We will see that the coordinates of the affine space $\mathbb{A}^N$ will correspond to coefficients of polynomials in $K[y]$. 
 The  parametrization allows us to find a formula for the dimension of this affine space in terms of the Hilbert function of $R/I_0$. As we will show in Section 6, this dimension coincides with the dimension of the Hilbert function strata in the Hilbert scheme of points in $\mathbb{P}^2$.  Other dimension formulas for this variety  were originally found in \cite{ES1,ES2,got,IK,NvdB}.  The  combinatorial nature of  our proof of the formula allows us to find as a corollary  the same  upper and lower bounds  for this dimension, that the authors find in \cite{NvdB}.

Unfortunately parts of the proof of the main theorem are rather technical. In order to clarify the most complicated steps we dedicate the fifth section to examples.

In Section 6 we will consider ideals of the polynomial ring in three variables, $S = K[x,y,z]$. For a monomial ideal  $J_0 \subset S$   it is known that  the affine variety $\V{J_0}$ is in general not an affine space (see \cite{iarro2, concavalla} for examples). We will assume that $J_0 = I_0S$, with $I_0$ a  lex-segment ideal of $K[x,y]$, and parametrize the variety $\V{J_0}$ with respect to the degree reversed-lexicographic term order induced by $x>y>z$. This will be an affine space of the same dimension as $V(I_0)$. 

In the last section we come to study the Betti strata of $\V{J_0}$, with $J_0 \subset S$ as above. We emphasize one more time that $\V{J_0}$ is a dense subset of  $\Hilb^H(\PR^{n-1})$ under some restrictions on the characteristic, where $H$ is the Hilbert function of $S/J_0$. 
In \cite[Remark 3.7]{iarro} A. Iarrobino gives  a generalization to codimension two punctual schemes in $\mathbb{P}^2$ of the codimension formula of the Betti strata, together with an indication of a proof.
Using the extension of the parametrization from Section 6, we obtain a different proof for the above-mentioned formula. 

The results of this paper were discovered and double-checked by extensive computer algebra experiments performed with CoCoA \cite{cocoa}.

The author wishes to thank his advisor, Prof. Aldo Conca, 
for suggesting this problem and for his very helpful remarks
on preliminary versions of this paper. Many thanks also to Prof. Giuseppe Valla for his useful remarks regarding the dimension formula of the Hilbert scheme. The author also thanks the anonymous referee for all the useful comments which helped improve the presentation of the paper.

\section{Preliminaries}

All the initial terms and ideals will be considered from now on with respect to the degree reverse lexicographic term order induced by $x>y$.

Let $I_0 \subset R = K[x,y]$ be a monomial ideal with $\dim_{k}(R/I_0) < \infty$. We choose for $I_0$ the following set of generators:
\[ I_0 := (x^t, x^{t-1}y^{m_1}, \ldots, xy^{m_{t-1}}, y^{m_t}), \]
where  $t := \min\{ j : x^j \in I_0\}$, $m_0 = 0$ and  $m_i := \min\{ j : x^{t-i}y^j \in I_0\}$ for every $1 \le i \le t$. Notice that we have $0 = m_0 \le m_1 \le \ldots \le m_t$; so let us  define $d_i:= m_i - m_{i-1}$ for all $ i = 1, \ldots t$.
It is clear that, even if the above generators are not always minimal,  the ideal $I_0$ is uniquely  determined by the sequence of the $m_i$'s, so also by that of the $d_i$'s. It is easy to check that the lex-segment ideals correspond to the vectors $d$ for which $d_i > 0$ for every $1 \le i \le t$.
We consider the following matrix:
\begin{displaymath}
X = \left(
  \begin{array}{cccc}
   y^{d_1} &0 &\ldots &0 \\ 
     -x &y^{d_2} &\ldots &0  \\ 
     0 &-x &\ldots &0 \\ 
     \ldots &\ldots &\ldots &\ldots \\ 
       0 &0 &\dots &y^{d_t} \\ 
           0 &0 &\ldots &-x\\  
  \end{array}
           \right).
\end{displaymath}
This matrix is a Hilbert-Burch matrix for $I_0$, in the sense that its signed minors are $x^{t-i}y^{m_i}$, so they generate the ideal, and its columns generate their syzygy module. 

It is useful to consider also the corresponding degree matrix $U(I_0)= (u_{i,j})$. The entries of $U(I_0)$ are the degrees of the (homogeneous) entries representing a map of degree zero:
$$F_1=\bigoplus_{i=1}^{t}R(-t+i-m_i-1)  \longrightarrow F_0 = \bigoplus_{i=0}^{t}R(-t+i-1-m_{i-1}).$$
 Notice that $X$ is such a matrix. We have 
\[ u_{i,j} = i-j+m_j-m_{i-1},\textup{~ for ~}i=1,\ldots,t+1\textup{~ and~} j=1,\ldots,t.\]
Let $A$ be another $(t+1) \times t$ matrix, with entries in the polynomial ring in one variable $K[y]$, with the following property:
\begin{equation}\label{A}
 \deg(a_{i,j}) \le \left.\bigg\{ \begin{array}{ccc}
                                 \textup{Min}\{u_{i,j} -1,& d_i - 1\}&\textup{if}~ i \le j, \\
                               \rule{0pt}{3ex}  \textup{Min}\{u_{i,j} \phantom{-1..},& d_j - 1\}    &\textup{if} ~ i > j,\\
                                   \end{array}\right.
\end{equation}
where $i = 1,\ldots,t+1$ and $j = 1,\ldots,t$. We will denote by \AA{I_0} the set of all matrices that satisfy the above condition.  Let $b_{i,j}$ denote the minimum on the right hand side of (\ref{A}) and define $N$ as the following positive integer:
$$N = \sum_{b_{i,j} \ge 0} (b_{i,j}+1).$$
Notice  that $\AA{I_0} = \mathbb{A}^N$.
  In Section 4 we will compute an exact formula for $N$, when $I_0$ is a lex-segment ideal, depending on the Hilbert function of $R/I_0$.  

For $i = 1,\ldots, t+1$ and any $A \in \AA{I_0}$, denote by   $[X+A]_i$ the matrix obtained by deleting the $i$th row of the matrix $X+A$. For $i= 0,\ldots,t$ we define the polynomials: 
 \[f_i:=(-1)^i \det([X+A]_{i+1}).\]
Let  $\psi : \AA{I_0} \longrightarrow \I{I_0}$ be the map of affine varieties  defined by
\[\psi(A) := I_t(X+A),\]
where by $I_t(X+A)$ is the ideal generated by $t$-minors of the matrix $X+A$. In particular 
$\psi(A)$ is the ideal generated by $f_0,\ldots,f_t$.

\section{Main theorem}

\begin{theorem}\label{main}
Let $I_0 \subset R = K[x,y]$ be a lex-segment ideal with $\dim_K(R/I_0) < \infty$. Then, the map of affine varieties $\psi: \AA{I_0} \longrightarrow \I{I_0}$ is bijective.
\end{theorem}
This will be the parametrization of \I{I_0} that we are looking for. 
To prove this we have to prove three things:
\begin{enumerate}
\item 
The application $\psi$ is well defined.
\item
The application $\psi$ is injective.
\item
The application $\psi$ is surjective.
\end{enumerate}

We believe the result to be true without the assumption that $I_0$ is a lex-segment ideal. However,  we were not able to prove the second point without this hypothesis. Hoping that such a proof exists and, as   $\psi$  is well defined and surjective in general (so it is a parametrization), we present here proofs of  the first and the third point that work  for any monomial ideal $I_0$ with  \mbox{$\dim_K(R/I_0) < \infty$.}
\subsection{Proof of 1}
We want to prove that
$\ini(I_t(X+A)) = I_0, \textrm{ for all~} A \in \AA{I_0}$. 
 By construction we have  $\ini(f_i) = x^{t-i}y^{m_i}$. We will show that $\{f_0, \ldots, f_t\}$ form a Gr$\ddot{\textrm{o}}$bner basis.

As the syzygy module of $I_0$ is generated by the columns of the matrix $X$, by an optimization of the Buchberger algorithm (see \cite{krobb}, Remark 2.5.6), 
we only  have to look at the  S-polynomials of the form 
\[y^{d_i } f_{i-1} - xf_{i}, \quad \textrm{for all } i = 1,\ldots,t, \]
and check that they can be written as $\sum_{j=0}^{t}P_j f_j$, with
\[ \ini(P_j f_j) \le \ini(y^{d_i } f_{i-1} - xf_{i}), \textrm{for all}~ j = 0,\ldots,t.\]
Again by construction we have that
\[y^{d_i } f_{i-1} - xf_{i} + \sum_{j=0}^{t} a_{j+1,i} f_j = 0.\]
As all the $a_{i,j}$ are polynomials in $K[y]$ and all the leading terms of the $f_j$'s are divisible by different powers of $x$, we get that the leading terms of the $a_{i,j}f_j$'s cannot cancel each other, so we must have
\[ \textup{Max}_{j}\{ \ini(a_{i,j}f_j) | a_{i,j} \neq 0\} = \ini(y^{d_i } f_{i-1} - xf_{i}).\]
Thus,  the application $\psi$ is well defined.
\subsection{Proof of 2}
Suppose  there exist  two matrices $A, B \in \AA{I_0}$ such that $I_t(X+A) = I_t(X+B)$. We want to prove that in this case $A=B$.
For $i = 0,\ldots, t$ denote
\begin{eqnarray}\label{fi}
f_i &:= & (-1)^i\det([X+A]_{i+1}),\\
 g_i&:= & (-1)^i\det([X+B]_{i+1}).
 \end{eqnarray}
 We will prove that  $f_i = g_i,~\forall~ i = 0,\ldots,t$. This implies that $A=B$, because  the columns of the matrix $(X+A)$ and $(X+B)$ are syzygies for the $f_i$'s, respectively the $g_i$'s. So, if $f_i = g_i$ for all $i$, then  also the columns of $(X+A)-(X+B) = A-B$ will be again syzygies. But the entries of $A-B$ are polynomials in $K[y]$, and as the leading terms of the $f_i$'s involve different powers of $x$, they must all be zero. 

 We will first prove  two lemmas.
Throughout  this section we will use the following notation. The entries of the matrices  $X+A$ and $X+B$  are denoted  by $\a_{i,j}$  respectively   by $ \b_{i,j} $. The entries of $A$ and  $B$ are denoted by $a_{i,j}$, respectively  $b_{i,j}$. The $\a_{i,j}$'s are of the following form:
\begin{displaymath}
\a_{i,j} = \left.\Bigg\{ \begin{array}{lll}
                    y^{d_i} + a_{i,i} & \textup{if}~ i = j,\\ 
                    -x + a_{i+1,i} & \textup{if}~ i = j+1,\\
                    a_{i,j} & \textup{otherwise}.\\
                    \end{array}\right.                    
\end{displaymath}
The $\b_{i,j}$'s have an analogous form.
First we will show that the homogeneous component of maximum degree of $f_i$ is equal to the homogeneous component of maximum degree of $g_i$, for all $i = 0,\ldots,t$. 

Let $\omega$ be the weight vector $(1,1)$. For a polynomial $f \in K[x,y]$ we denote by $\ini_{\omega}(f)$ the sum of the monomials of maximum degree. For an ideal $I \subset K[x,y]$ we denote by $\ini_{\omega}(I) := \langle \ini_{\omega}(f) ~:~ f \in I \rangle$. We will prove the following lemma:
\begin{lemma}\label{inj1}
Let $I_0 \subset R$ be a monomial ideal with $\dim_K(R/I_0) < \infty$ and  let $A, B \in \AA{I_0}$ be two matrices such that $I_t(X+A) = I_t(X+B) = I$. Then, with the above notations, we have:
\[ \ini_{\omega}(f_i) = \ini_{\omega}(g_i), \quad \forall~i = 0,\ldots,t.\]
\end{lemma} 
\begin{proof}
As the DRL order is a refinement of the partial order given by the weight vector $\omega$, we have by \cite{sturmfels} that:
\[ \ini(\ini_{\omega}(I)) = \ini(I) = I_0.\]
By the proof of 1. we know  that both $\{ f_i \}_{i=0,\ldots, t}$ and $\{ g_i \}_{i=0,\ldots, t}$ are \GBs~ with respect to the DRL term order. Again by \cite{sturmfels}
we get that $\{ \ini_{\omega}(f_i) \}_{i=0,\ldots, t}$ and $\{ \ini_{\omega}(g_i) \}_{i=0,\ldots, t}$ are  DRL \GBs~of $\ini_{\omega}(I)$.

As the bounds on the degrees of the entries of $A$ and $B$ are connected to $U(I_0)$, so to  the degrees of homogenous matrices, we have that the homogeneous polynomials $\ini_{\omega}(f_0),\ldots,\ini_{\omega}(f_t)$ will be the maximal minors of a matrix $X+A'$, where the entries $a'_{i,j}$  of $A'$ have the following property:
\[ a'_{i,j} = \left.\bigg\{ \begin{array}{ll}
                                c_{i,j}y^{u_{i,j}} &\textup{if}~ j<i ~\textup{and}~ 0 \le u_{i,j} < d_j, \\
                               \rule{0 pt}{3ex} 0     &\textup{otherwise}, \\
                                   \end{array}\right. \]
with $c_{i,j} \in K$.
The same holds for the polynomials $\ini_{\omega}(g_0),\ldots,\ini_{\omega}(g_t)$. Let us say they are the maximal minors of a matrix $X+B'$, with $B'$ having the same property as $A'$.

As all these polynomials are homogeneous elements of $K[x,y]$, their leading term is the same for the DRL and the Lex term order. So we find ourselves in the case already solved  in \cite{concavalla}. That is the matrices $A'$ and $B'$ parametrize the same homogeneous ideal, $\ini_{\omega}(I)$. So, by \cite[Theorem 3.3]{concavalla},  they must be equal. Thus we also have that $\ini_{\omega}(f_i) = \ini_{\omega}(g_i)$ for all $i = 0,\ldots, t$. 
\end{proof}

If we denote by $I=I_t(X+A)$, we have the following lemma:
\begin{lemma}\label{lema1}
Let $f \in I$ be a polynomial such that $x^t$ does not divide any monomial $m \in \textup{Supp}(f)$ and   let $f_i$ be the polynomials defined in (\ref{fi}) for any $i \in \{0,\ldots,t\}$. Then $f$ can be written as:
\[ f = \sum_{i = 1}^{t} P_{i} f_i,\]
with $P_i \in K[y]$ and $\deg(f_i) \le \deg(f)$.
\end{lemma}
\begin{proof}
We have that $\ini(f) = x^sy^r$ with $s < t$. As $\ini(f) \in \ini(I)$, we have that $r \ge m_{t-s}$. We now define a new polynomial: 
\[f' := f - \textup{LC}(f)y^{r-m_{t-s}} f_s,\]
where LC$(f)$ is the leading coefficient of $f$. By construction, the monomials that appear in the support of the $f_i$'s are not divisible by $x^t$ for $i = 1,\ldots,t$. So we have that $f'$ has the same property as $f$. After a finite number of steps we will obtain the desired representation. 
\end{proof}
It is important to keep in mind from Lemma \ref{lema1} that $f_0$ is not needed to rewrite $f$ and that the polynomials $P_i$ are only in the variable $y$. 
The above lemmas do not need the assumption that $I_0$ is a lex-segment ideal.\\

 \emph{Proof of the Injectivity.} 
From this point on, we will use the fact that $I_0$ is a lex-segment ideal, i.e. for any monomial $u \in I_0$ of degree $d$, all the monomials $v$ of degree $d$, with $v >_{\Lex} u$ are also in $I_0$.
The idea of the following proof is to rewrite $f_i$'s in terms of the $g_i$'s, to put the appearing coefficients in a matrix $R$ and then prove that this matrix is actually the identity matrix. The lex-segment hypothesis will allow us to do this  \lq\lq block-wise\rq\rq ~with respect to $R$. For any monomial ideal $I_0$ the shapes of the matrices $A$, $B$ and $R$ are difficult to control thus a \lq\lq block-wise\rq\rq or inductive proof is not known to us in the general setting. However, we will see in the next sections that this hypothesis is not so restrictive, meaning that the generic case is the lex-segment case when the characteristic of $K$ is zero or \lq\lq large enough\rq\rq.

Recall that $I_0$ is a lex-segment ideal if\mbox{}f $d_i > 0, ~\forall~ i = 1,\ldots,t$. This means that $\deg(f_{i-1})\le \deg(f_{i})$ for all $1 \le i \le t$. The indices $i$ for which the inequality is strict and  those for which the \lq\lq jump\rq\rq ~in the degree of the generators is higher than 1 play an important role in the proof.  
Let us denote the two sets of  these \lq\lq special\rq\rq~ indices by
\begin{eqnarray*}
 \mathcal{J}& := & \{ j \in 1, \ldots , t ~:~ d_j \ge 2 \}, \\
  \mathcal{I} &:= & \{ i \in 1,\ldots, t ~:~ d_i \ge 3\}. 
 \end{eqnarray*}
 Let $ i_1 \ldots, i_q$ be the elements of $\mathcal{I}$, respectively   $j_1,\ldots,j_p$ the elements of  $\mathcal{J}$, in increasing order. To simplify statements we will consider also $i_0=j_0=1$ and $i_{q+1}=j_{p+1}=t+1$. 
The fact that $I_0$ is a lex-segment allows us to give a more accurate description of maximal degrees that may appear in $X+A$. 
Above the diagonal ($i\le j$) we will have the following bounds:
\[ \textrm{Min}\{i-j+m_j - m_{i-1} -1, d_i -1\}=\textrm{Min}\{i-j+\sum_{k=i}^{j}d_k -1 , d_i -1\} = d_i - 1.\]
Below the diagonal ($i>j$) things are slightly more complicated:
\[ i-j +m_j - m_{i-1} =  i - j -  d_{j+1} - \ldots -   d_{i - 1}\le 1.\]
We can divide the matrix $X + A$ into blocks depending on the indices in $ \mathcal{I}$ as follows:
 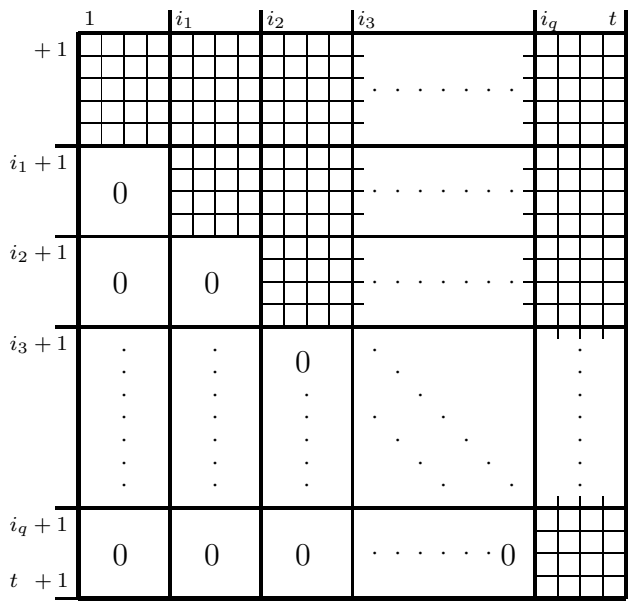
\begin{figure}[!h]
\setlength{\unitlength}{1.5mm}
\begin{picture}(52,54)
\linethickness{0.5mm}      
\multiput(4,0)(0,50){2}{\line(1,0){48}}
\multiput(4,0)(48,0){2}{\line(0,1){50}}

\linethickness{0.3mm}

\multiput(2,0)(0,8){2}{\line(1,0){50}}     
\multiput(2,24)(0,8){3}{\line(1,0){50}}

\multiput(12,0)(8,0){3}{\line(0,1){52}}       
\multiput(44,0)(8,0){2}{\line(0,1){52}}

\linethickness{0.1mm}

\multiput(4,42)(0,2){4}{\line(1,0){25}}  
\multiput(43,42)(0,2){4}{\line(1,0){9}}

\multiput(12,34)(0,2){3}{\line(1,0){17}}
\multiput(43,34)(0,2){3}{\line(1,0){9}}

\multiput(20,26)(0,2){3}{\line(1,0){9}}
\multiput(43,26)(0,2){3}{\line(1,0){9}}

\multiput(44,2)(0,2){3}{\line(1,0){8}}

\multiput(6,40)(2,0){3}{\line(0,1){10}}  

\multiput(14,32)(2,0){3}{\line(0,1){18}}

\multiput(22,24)(2,0){3}{\line(0,1){26}}

\multiput(46,0)(2,0){3}{\line(0,1){9}}
\multiput(46,23)(2,0){3}{\line(0,1){27}}

{\Large                
\put(7,3){$0$}
\put(15,3){$0$}
\put(23,3){$0$}
\put(7,27){$0$}
\put(15,27){$0$}
\put(7,35){$0$}
\put(23,20){$0$}
\put(41,3){$0$}
}
{\footnotesize
\put(-2,1){$t\phantom{_1}+1$} 
\put(-2,6){$i_q+1$}
\put(-2,22){$i_3+1$}
\put(-2,30){$i_2+1$}
\put(-2,38){$i_1+1$}
\put(-2,48){$\phantom{i_q}+1$}

\put(4.5, 50.66){$1$}         
\put(12.5, 50.66){$i_1$}
\put(20.5, 50.66){$i_2$}
\put(28.5, 50.66){$i_3$}
\put(44.5, 50.66){$i_q$}
\put(50.5, 50.66){$t$}
}
\multiput(8,10)(0,2){7}{\circle*{0.3}}  
\multiput(16,10)(0,2){7}{\circle*{0.3}}
\multiput(24,10)(0,2){5}{\circle*{0.3}}
\multiput(48,10)(0,2){7}{\circle*{0.3}}

\multiput(30,4)(2,0){6}{\circle*{0.3}}  
\multiput(30,28)(2,0){7}{\circle*{0.3}}
\multiput(30,36)(2,0){7}{\circle*{0.3}}
\multiput(30,45)(2,0){7}{\circle*{0.3}}

\multiput(30,22)(2,-2){7}{\circle*{0.3}} 
\multiput(30,16)(2,-2){4}{\circle*{0.3}}

\end{picture}
\caption{The matrix $X+A$}\label{allA}
\end{figure}\\
Now we will take a closer look at the nonzero part of this matrix. Let $\a \in 0, \ldots, q$. Denote by $\{j_{\b},\ldots , j_{\b+k}\} := \{ j \in  \mathcal{J} ~:~ i_{\a} < j < i_{\a +1}\}.$ It is possible that this set is empty, which only simplifies the picture. If the set is not empty,   the index $\b$ depends on $\a$.
The matrix formed by the first $i_{\a +1}$ rows of the columns indexed from $i_{\a}$ to $i_{\a +1} -1$ has the following form:
 \begin{figure}[!h]

\begin{center}

\setlength{\unitlength}{2.5mm}
\begin{picture}(26,30)
\linethickness{0.5mm}      
\multiput(2,0)(0,29){2}{\line(1,0){20}}
\multiput(2,0)(20,0){2}{\line(0,1){29}}

\linethickness{0.3mm}

\put(2,20){\line(1,0){4}}     
\put(14,4){\line(1,0){8}}
\put(6,8){\line(1,0){4}}
\put(2,12){\line(1,0){8}}
\put(2,16){\line(1,0){8}}

\multiput(14,0)(4,0){2}{\line(0,1){4}}        
\put(6,8){\line(0,1){12}}
\put(10,8){\line(0,1){8}}

\linethickness{0.1mm}

\multiput(2,28)(0,-1){8}{\line(1,0){20}}  
\put(1,20){\line(1,0){1}}

\multiput(6,20)(0,-1){4}{\line(1,0){16}}

\multiput(10,16)(0,-1){4}{\line(1,0){12}}

\put(14.5,11){\line(1,0){7.5}}
\put(14.5,10){\line(1,0){7.5}}
\put(14.5,9){\line(1,0){7.5}}
\put(14.5,8){\line(1,0){7.5}}
\put(18,7){\line(1,0){4}}
\put(18,6){\line(1,0){4}}
\put(18,5){\line(1,0){4}}

\linethickness{0.05mm}
\multiput(6,0)(4,0){2}{\line(0,1){30}}
\multiput(18,0)(4,0){2}{\line(0,1){30}}
\multiput(1,4)(0,8){2}{\line(1,0){21}}
\multiput(1,12)(0,4){2}{\line(1,0){1}}
\put(1,0){\line(1,0){1}}
\put(2,29){\line(0,1){1}}

{\Large                
\put(3.5,1.5){$0$}
\put(7.5,1.5){$0$}
\put(3.5,9.5){$0$}
\put(12.5,1.5){$0$}
}
{\Large                
\put(3.5,13.5){C}
\put(7.5,9.5){C}
\put(15.5,1.5){C}
}
{\Large                
\put(3.5,17){*}
\put(7.5,13){*}
\put(19.5,1){*}
}

\linethickness{2.5mm}
\multiput(2,20.5)(4,-4){3}{\line(1,0){1}}
\put(18,4.5){\line(1,0){1}}

{\footnotesize
\put(-1,0.5){$i_{\a+1} $} 
\put(-1,4.5){$j_{\b+k}$}
\put(-1,12.5){$j_{\b+1}$}
\put(0.25,16.5){$j_{\b}$}
\put(0.25,20.5){$i_{\a}$}

\put(2.25, 29.5){$i_{\a}$}         
\put(6.25, 29.5){$j_{\b}$}
\put(10.25, 29.5){$j_{\b+1} ~~~~ \ldots$}
\put(18.25, 29.5){$j_{\b+k}$}
\put(22.25, 29.5){$i_{\a+1} $}

}
\multiput(4,5)(0,1){3}{\circle*{0.15}}  
\multiput(8,5)(0,1){3}{\circle*{0.15}}

\multiput(9.5,2)(1,0){3}{\circle*{0.15}}  

\multiput(11,7)(1,-1){3}{\circle*{0.15}} 
\multiput(17,5)(-1,1){7}{\circle*{0.15}}

\end{picture}

\end{center}
\caption{ The first $i_{\a +1}$ rows of the columns indexed from $i_{\a}$ to $i_{\a +1} -1$ in $X+A$}\label{partA}
\end{figure}\\
In order to be able to present the matrices in a more compact and suggestive way,  we denote by $ {\scriptstyle\bullet} y^s$, for $s > 0$,  a polynomial in $K[y]$, with $\deg({\scriptstyle\bullet}y^s) \le s$.  That is 
$$ {\scriptstyle\bullet} y^s = \sum_{i=0}^{s} a_iy^i, \quad\textup{ where } a_i \in K~\forall ~0\le i\le s.$$
In   Figure \ref{partA} the part above the diagonal of the $k$th row consists only of $ {\scriptstyle\bullet} y^{d_k-1}$. In particular, if $k \notin \mathcal{J}$, then it is made of constants. The blocks denoted by C consist also only of constants. The black squares, which correspond to elements on the diagonal in the position $(j,j)$ with $j \in \mathcal{J}$ are $y^{d_j} +  {\scriptstyle\bullet} y^{d_j -1}$.
 Finally, all matrices denoted by {\Huge \raisebox{-0.95 ex}{*}} are of the form:
\[
{\textrm{\Huge \raisebox{-0.95 ex}{*}}}= \left( 
  \begin{array}{ccccc}
    -x +  {\scriptstyle\bullet} y &\phantom{-}y+c & c & \ldots & c \\ 
     {\scriptstyle\bullet} y & -x + c & \phantom{-}y+c & \ldots & c \\ 
     {\scriptstyle\bullet} y & c & -x + c& \ldots & c \\ 
    \vdots& \vdots & \vdots & \ddots & \phantom{-}y+c \\ 
     {\scriptstyle\bullet} y & c & c & \ldots & -x + c \\ 
  \end{array}
  \right).
\]
 For simplicity, we denote by $c$ a constant in general, so all $c$'s that appear can be different one from the other, and can also be zero.

By Lemma \ref{inj1} we have that $\deg(f_i-g_i)< \deg(f_i)$. By applying  Lemma \ref{lema1} to $f_i-g_i$ with respect to the $g_i$'s, we can write for any $i= 0,\ldots,t$
\begin{equation}\label{rewrite}
 f_i = g_i + \sum_{\deg(g_j) < \deg(f_i)} R_{j,i} \, g_j, \qquad~~~ \textrm{with } R_{j,i} \in K[y],~\forall~ i,j.
 \end{equation}
So we can form a $(t+1)\times(t+1)$ transition matrix $R$, with entries in $K[y]$, such that:
\[ (g_0,\ldots,g_t) R = (f_0,\ldots,f_t).\]
As the $f_i$'s and $g_i$'s are indexed from $0$ to $t$, we also indexed the rows and columns of the matrix $R$ starting from $0$ for simplicity. Notice that from Lemma \ref{lema1} we have that $R_{0,j} = 0, ~\forall~ 1\le j\le t$.

At this point the  choice of $I_0$ as a lex-segment ideal comes again  into play. Namely, as we already mentioned, $I_0$ is a lex-segment ideal if and only if 
\[\deg(g_i)=\deg(f_i)\le\deg(f_{i+1})=\deg(g_{i+1}), \forall~0\le i < t.\]
Together with \eqref{rewrite} we obtain that $R_{i,i} =1$ for all $0\le i \le t$ and $R_{i,j}=0$ for $0\le j < i \le t$. Furthermore, as  $j\in\mathcal{J}$ is equivalent to $\deg(f_j)<\deg(f_{j+1})$, the matrix $R$ can be divided into blocks depending on the indices $\{j_1,\ldots,j_p\}= \mathcal{J}$ in the following way:
\begin{figure}[!h]
\setlength{\unitlength}{4.5mm}
\begin{picture}(23,22)
\linethickness{0.5mm}      
\multiput(1,0)(0,21){2}{\line(1,0){21}}
\multiput(1,0)(21,0){2}{\line(0,1){21}}

\linethickness{0.3mm}

\put(1,4){\line(1,0){21}}     
\multiput(1,12)(0,4){3}{\line(1,0){21}}

\multiput(6,0)(4,0){2}{\line(0,1){21}}  
\put(14,0){\line(0,1){4}}
\put(14,12){\line(0,1){9}}      
\put(18,0){\line(0,1){21}}

\linethickness{0.05mm}
\put(14,4){\line(0,1){8}}
{\Huge                
\put(3,1.5){$0$}
\put(7.5,1.5){$0$}
\put(11.5,1.5){$0$}
\put(3,13.5){$0$}
}
{\huge                
\multiput(6.66,12.5)(4,-4){1}{$0$}
\multiput(8.66,14.5)(4,-4){1}{$0$}
\put(4.66,18.5){0}
\put(19,1){0}
\put(20.5,2.5){0}
\put(1.66,17){0}
}
\put(2.5,20.25){$0 ~~\ldots$}
\put(6.5,20.25){$0 ~~\ldots$}
\put(10.5,20.25){$0 ~~\ldots$}
\put(18.5,20.25){$0 ~~\ldots$}

\put(1.5,20.25){1}
\put(2.5,19.25){1}
\put(5.5,16.25){1}
\put(6.5,15.25){1}
\put(9.5,12.25){1}
\put(18.5,3.25){1}
\put(21.5,0.25){1}
{\footnotesize
\put(6.25,19.25){$R_{1,j_1} ~\ldots$}
\put(10.25,19.25){$R_{1,j_2} ~\ldots$}
\put(19.15,19.25){$ \ldots ~R_{1,t}$}

\put(6.25,16.33){$R_{j_1-1,j_1} \ldots$}
\put(10.25,16.33){$ R_{j_1-1,j_2} \ldots$}
\put(18.15,16.33){$  \ldots~R_{j_1-1,t}$}

\put(10.25,12.33){$R_{j_2-1,j_2} \ldots$}
\put(18.15,12.33){$  \ldots~R_{j_2-1,t}$}

\put(18.15,4.33){$  \ldots~R_{j_p-1,t}$}
}
{\footnotesize
\put(0.5,0.33){$t $} 
\put(0.25,3.33){$j_p$}
\put(0.25,11.33){$j_2$}
\put(0.25,15.33){$j_1$}
\put(0.25,19.33){$1$}

\put(1.25, 21.33){$0$}         
\put(6.25, 21.33){$j_1$}
\put(10.25, 21.33){$j_2$}
\put(14.25, 21.33){$\quad~~~~ \ldots$}

\put(18.25, 21.33){$j_p$}
\put(21.5, 21.33){$t $}
}

\multiput(3.5,5)(0,1){7}{\circle*{0.15}}  
\multiput(8,5)(0,1){7}{\circle*{0.15}}
\multiput(20,6)(0,1){6}{\circle*{0.15}}

\multiput(8,18.5)(0,-0.5){3}{\circle*{0.1}}
\multiput(12,18.5)(0,-0.5){3}{\circle*{0.1}}
\multiput(12,14.5)(0,-0.5){3}{\circle*{0.1}}
\multiput(20,18.5)(0,-0.5){3}{\circle*{0.1}}
\multiput(20,14.5)(0,-0.5){3}{\circle*{0.1}}

\multiput(15,2)(1,0){3}{\circle*{0.15}}  
\multiput(15,14)(1,0){3}{\circle*{0.15}}
\multiput(15,18)(1,0){3}{\circle*{0.15}}

\multiput(3.5,18.75)(0.75,-0.75){3}{\circle*{0.1}} 
\multiput(7.5,14.75)(0.75,-0.75){3}{\circle*{0.1}}
\multiput(11,11)(1,-1){7}{\circle*{0.15}}
\multiput(19.5,2.75)(0.75,-0.75){3}{\circle*{0.1}}

\end{picture}
 \caption{The transition matrix $R$}\label{allR}
 \end{figure}\\
 \newpage
\noindent In order to  prove injectivity we want to show that  $R$ is the $(t+1)\times(t+1)$ identity matrix. We will see in the paragraphs below that the particular shapes of the matrices $A$, $B$ and $R$ (which are given by the fact that $I_0$ is the lex-segment ideal) will allow us to prove this block by block. 
The number of blocks will be $q+1$ where $q$ is the cardinality of the index set $\mathcal{I}$.
These blocks are $(t+1)\times(i_{a+1}-i_a)$ submatrices of $R$ obtained by taking the columns indexed by $i$, with $i_a\le i \le i_{a+1}-1$, for all $0\le a \le q$. Notice that the first column is always the transposed vector $(1,0,0,\ldots,0)$ of length $t+1$. This is why we do not consider it as part of any block.

 Here is the plan for the next and final part of the proof. 
 Notice that the columns of the matrix $R(X+A)$ are syzygies for the $g_i$'s. We subtract from these syzygies appropriate multiples of the columns of $X+B$ such that we obtain new syzygies of the $g_i$'s, this time with entries in $K[y]$. So all the entries must be actually 0. These entries will be linear combinations of the $R_{i,j}$'s, with coefficients the entries of $A$ and $B$ and some $y^{d_i}$. Given the restrictions on the degrees of the entries in $A$ and $B$ we will deduce some limitations on the degrees of the nonzero $R_{i,j}$'s. This will be the most technical part of the proof. In the end we will show that these bounds lead to a contradiction, so all $R_{i,j}$ with $i\neq j$ are 0. \\

We will start with the first block, i.e. with the matrix formed by the columns of  $R$ indexed from $i_0=1$ to $i_1-1$. During the proof we will point to how the induction works and why the proof for the first block is sufficient (see Remark \ref{induction}).  

Notice that if $i_1=j_1$, the first two blocks are already a submatrix of the identity matrix. This fact is easy to see in Figure \ref{allR}. Also, if $j_1=1$ and $j_2=i_1$  we obtain that the first block is degenerated (has zero rows).  For such blocks there is nothing to prove, so   we may assume for the proof that this situations do not occur. 

Fix   $i_0 \le s \le i_1-1$ and denote by $E'_s$ the  $s$th column of $R(X+A)$. This will be a syzygy for the $g_i$'s with  entries of the form:
$$E'_{r,s}= \sum_{k=0}^{t} R_{r-1,k}\a_{k+1,s}, $$
where $1 \leq r \leq t+1$. From the shape of the matrix $X+A$ (see Figure \ref{allA}) we can see that $\a_{k+1,s}=0$ if $k+1\ge i_1+1$. From the shape of the matrix $R$ (Figure \ref{allR}) we see that if $r>k+1$ then $R_{r-1,k}$ is already 0. So the only possibly nonzero $R_{r-1,k}$'s that actually appear in these first $i_1-1$ columns are the ones of the first block. This means that the range of the indices is: $r-1 \in \{0,\ldots,i_1 -2\}$ and $k \in \{i_0,\ldots,i_1-1\}$. Notice that when $r=1$ the entry is actually $\a_{1,s}.$   In particular the entries of $R(X+A)$ that are of  interest  have also  $2 \leq r \leq i_1$: 
\[E'_{r,s} =  \a_{r,s} + \sum_{k=i_0}^{i_1-1} R_{r-1,k}~ \a_{k+1,s},\]

As $\a_{s+1,s} = -x + \ldots$, to cancel out  the $x$'s in every entry of $R(X+A)$ we must subtract from this new syzygy:  $R_{r-1,s} \times $ (the $(r-1)$th column of $X+B$). We do this  for every $r = 2, \ldots, i_1$. So we obtain a new syzygy for the $g_i$'s, which we denote by $E_s$, with the following entries:
\[E_{r,s} = \a_{r,s} +  \sum_{k=i_0}^{i_1-1} R_{r-1,k}~\a_{k+1,s} -\sum_{l=1}^{i_1-1} R_{l,s}~ \b_{r,l}. \]
 Note  that $-x$ appears first with coefficient $R_{r-1,s}$ and then in the second sum also with coefficient $R_{r-1,s}$, we can conclude that each entry is a polynomial in $K[y]$. As we just added and subtracted syzygies, we obtain again a syzygy. But as in the initial terms of the $g_i$'s there appear different powers of $x$, we get that all entries must be 0. So we have the following equations:
\[ E_{r,s} = 0.\] 
We will interrupt the proof in order to make the following important remark:
\begin{remark}\label{induction}
We will explain here how the block by block proof works. In general, by the above arguments, when we look at the columns of $R(X+A)$ indexed from $i_a$ to $i_{a+1}-1$, the  entries that may be different from $\a_{r,s}$ involve  only the $R_{k,l}$'s of the first $a+1$ blocks. This means that also the equations $E_{r,s}=0$ involve only $R_{k,l}$'s from the first $a+1$ blocks.
Assume that the first $a$ blocks are already in the desired form ($R_{k,l}\neq 0 $ if\mbox{}f $k=l$). Then for $i_a\le s\le i_{a+1}-1$ the equations $E_{r,s}$ involve  only the $R_{k,l}$'s of the $(a+1)$-th block. With the inductive hypothesis the proof for the $(a+1)$-th block is analogous to the proof for the first block. One just needs to replace $i_0$ by $i_a$ and $i_1$ by $i_{a+1}$. For this reason we will present only the proof for the first block.
\end{remark}
Let us now resume the proof.
We are considering any $A, B \in \AA{I_0}$, so some of the $\a_{i,j}$'s and $\b_{i,j}$'s may be 0. 
This is the reason why we will often use the expression \lq\lq may have degree\rq\rq ~  instead of \lq\lq has degree\rq\rq.
 But we will always have that $\deg(\a_{ii}) = \deg(\b_{ii}) = d_i$. 
We can present  the equations that we obtained so far in a more compact way. 
The entries of the matrix $R$ below the diagonal and some of the ones above are already 0. For any indices $r$ and $s$ denote by:
\begin{eqnarray*}
 j(r) &:= &\Max\{ i ~:~ \deg(f_i)=\deg(f_r)\} + 1 = \Min\, \{ j \in \mathcal{J}~:~ j > r\}, \\
 \widetilde{j}(s) &:= & \Min\,\{ i ~:~\deg(f_i)=\deg(f_s)\}\phantom{-11}=  \Max \{ j \in \mathcal{J}~:~ j \le s\}. 
 \end{eqnarray*}
We can easily notice that  the equations $E_{r,s}$  actually are:
\[E_{r,s} =   a_{r,s} - b_{r,s} +  \sum_{k=j(r-1)}^{i_1-1} R_{r-1,k}~ \a_{k+1,s} -\sum_{l=1}^{\widetilde{j}(s)-1} R_{l,s}~ \b_{r,l} = 0.\] 
Recall that in this part we will show that $R_{r,s} = 0$ for $r < s$ with $1 \le r \le i_1$ and $i_0 \leq s \le i_1-1$. For each such pair $(r,s)$ we will deduce the limitations on the degree of $R_{r,s}$ from the equation $E_{r,s} = 0$. We will be able to do this, because the coefficient of $R_{r,s}$ will have maximal degree among the other coefficients that appear in $E_{r,s}$. So if $R_{r,s} \neq0$ there must exist  another $R_{k,l}$ with higher or equal degree.  We will order the $R_{r,s}$ and prove inductively that no $R_{r,s} \neq 0$ can have maximal degree, which means that all of them must actually be 0.
 
Depending on $r$ and $s$ there are four types of equations. From the first three types we can deduce directly upper bounds on the degree of  $R_{r,s}$. The fourth type may need to be modified in order to obtain such bounds.  \\
{\bf Type 1}:   If ($s \notin \mathcal{J}$ and $r \in \mathcal{J}$), or ($r \in \mathcal{I} \cup \{1\} $) then, as ($d_r \ge 2$ and $d_s =1$) or ($d_r \ge 3$ and $d_s \le 2$) we get:
\[
\deg(R_{r,s}) < \left\{ \begin{array}{lll}
  \deg(R_{r-1,k})& \textrm{~for some~} k \in \{r,\ldots, i_1-1\},&\quad \textrm{or}\\
\rule{0pt}{3ex} \deg(R_{l,s})& \textrm{~for some~}l \in \{1,\ldots,s-1\}, ~l \neq r.&
\end{array} \right.\\
\]
Notice that because $E'_{1,s} = \a_{1,s}$, we have that  $E_{1,s}$ is of   this type.\\
{\bf Type 2}:  If $s \notin \mathcal{J}$ and $r \notin \mathcal{J}$  then, as $d_r = 1$ and $d_s = 1$, we get: 
\[\left\{
\begin{array}{l}
\deg(R_{r,s}) < \left\{ \begin{array}{lll}
  \deg(R_{r-1,k})& \textrm{~for some~} k+1 \notin \mathcal{J},&\textrm{or}\\
\rule{0pt}{3ex} \deg(R_{l,s})& \textrm{~for some~} r\neq l \notin \mathcal{J}, \textrm{~or~} l > r,&
\end{array} \right.\\

\rule{0pt}{3ex}\quad\textrm{or}\quad \\

\rule{0pt}{3ex}\deg(R_{r,s}) \le \left\{
\begin{array}{lll}
  \deg(R_{r-1,k})&\textrm{~for some~} k < s, \textup{with~}k+1 \in \mathcal{J}, &\quad \textrm{or}\\
\rule{0pt}{3ex} \deg(R_{l,s}) &\textrm{~for some~} l \in \mathcal{J}, \textrm{~and~} l  < r. &
\end{array} \right.
\end{array}\right.
\]
{\bf Type 3}: 
 If $s \in \mathcal{J}$ and $r \in \mathcal{J}\setminus \mathcal{I}$  then, as $d_r = 2$ and $d_s = 2$, we get:
 \[\left\{
\begin{array}{l}
\deg(R_{r,s}) < \left\{ \begin{array}{lll}
  \deg(R_{r-1,k})& \textrm{~for some~} k \neq s-1,\ &\quad \textrm{or}\\
\rule{0pt}{3ex} \deg(R_{l,s})& \textrm{~for some~} l \neq r,&
\end{array} \right.\\
\rule{0pt}{3ex}\quad\textrm{or}\quad \\
\rule{0pt}{3ex}\deg(R_{r,s}) \le \deg(R_{r-1,s-1}).
\end{array}\right.
\] 
{\bf Type 4}:
If $s \in \mathcal{J}$ and $r \notin \mathcal{J}$  then we have  $d_r = 1$ and $d_s = 2$. So in this case we need to modify the original equation, because $R_{r-1,s-1}$ has coefficient of maximal possible degree. As $\a_{s,s}$ is the coefficient of $R_{r-1,s-1}$ we can look at the equation $E_{r-1,s-1}=0$. We will assume by induction that whenever the equation $E_{r-i,s-i} = 0 $ is of type 4, it has already been brought to the desired form for all $i>0$ (i.e. with $R_{r-i,s-i}$ having coefficient of maximal degree). There are  two sub-cases:\\
If $E_{r-1,s-1}=0$ is of type 1, then we redefine 
\[ E_{r,s}:= y^{d_{r-1} - 2}E_{r,s} + E_{r-1,s-1}.\]
The new coefficient of $R_{r,s}$ is  $y^{d_{r-1} - 2}\b_{r,r}$  and has maximal degree as we wanted.  \\
If $E_{r-1,s-1}=0$ is not of type 1, then we redefine 
\[ E_{r,s}:= y^{c_1}E_{r,s} + y^{c_2} E_{r-1,s-1},\]
where $c_1= \max\{0, \deg(\g_{r-1,s-1})-2\}$ and  $c_2= \max\{0, 2 -\deg(\g_{r-1,s-1})\}$ and  $\g_{r-1,s-1}$ is the coefficient of $R_{r-1,s-1}$ in $E_{r-1,s-1}$. 
If we still did not obtain a coefficient of maximal degree for $R_{r,s}$, then the new $R_{k,l}$'s that have coefficient of maximal degree, are of the form $R_{k,l}$, with $k < r-1$. So by repeating this procedure we will reach at some point the previous case.

For this type of equations there are three kinds of conclusions that we can draw:
\[\left\{
\begin{array}{l}
\deg(R_{r,s}) < \deg(R_{k,l})  \textrm{~~~~~~~~~for some~} R_{k,l},\\
\rule{0pt}{3ex}\quad\textrm{or}\quad \\
\rule{0pt}{3ex}\deg(R_{r,s}) \le \left\{
\begin{array}{lll}
  \deg(R_{k,l})&\textrm{~for some~} l+1 \in \mathcal{J} \textrm{~and~} k < r, &\quad \textrm{or}\\
\rule{0pt}{3ex} \deg(R_{k,l}) &\textrm{~for some~} k <j(r-1) \textrm{~and~} l < s, &
\end{array} \right.\\
\rule{0pt}{3ex}\quad\textrm{or}\quad \\
\rule{0pt}{3ex}\deg(R_{r,s}) = 0.
\end{array}\right.
\]
We were vague for the strict inequality, because we do not need to know the indices in that case.
The third possibility comes from the fact that when performing the above operations, we may find that the degree of the coefficient of $R_{r,s}$ is equal to the degree of the free term.

Now we just have to see that these inequalities imply $R_{r,s} =  0$. 
To be able to conclude, we also need to order the $R_{k,l}$'s in the following way:
\[ R_{r,s} < R_{k,l} \Leftrightarrow 
\left\{ \begin{array}{lllll}
  \deg(f_r)<\deg(f_k)&&&&\textup{or}\\
    \deg(f_r)=\deg(f_k),&\deg(f_s)<\deg(f_l),&&&\textup{or}\\
      \deg(f_r)=\deg(f_k),&\deg(f_s)=\deg(f_l),& k<r,&&\textup{or}\\
  \deg(f_r)=\deg(f_k),&\deg(f_s)=\deg(f_l),& k=r, &s<l&\\    
  \end{array} \right.\]
The following remarks are the key to the last part of the proof. They are an immediate consequence of the above discussion.
Assume that there exists an $R_{r,s}\neq0$ with $r\neq s$ in this block.
\begin{remark}Denote by $M := \max\{ \deg(R_{r,s})~:~  R_{r,s}\neq0,  1 \le r \le i_1 \textrm{~and~} i_0 \leq s \le i_1-1 \}$. 
\begin{itemize}
\item[1.] If $R_{r,s}$ is the smallest element according to the order above, then $E_{r,s}= 0$ is of type 1.
\item[2.] 
If $E_{r,s}=0$ is of type 1, then  $\deg(R_{r,s}) < M$. 
\item[3.]
If $\deg(R_{r,s}) = M > 0$, then $\deg(R_{r,s}) = \deg(R_{k,l})$  with $R_{r,s} > R_{k,l}.$
\end{itemize}
\end{remark}
Because of the type 4 equations, we have to distinguish the following two cases.\\
\emph{Case 1:} $M > 0$. Choose the minimal $R_{r,s}$ such that $\deg(R_{r,s}) = M$. Then, by the above remark, we already obtain a contradiction.\\
 \emph{Case 2:} $M = 0$.
In this case we can use induction on $r$, ignoring the complicated order defined above. Also, we will not need to modify the equations of type 4.\\
If $r =1$ then we are in the type 1 situation. So $\deg(R_{1,j}) < M = 0$.\\
Suppose $R_{i,j} = 0$ for all $i < r$. Then for all four types of equations, when we replace with 0 the $R_{i,j}$'s  with $i < r$, we get equations of the form:
\[  a_{r,s} - b_{r,s} -\sum_{l=r}^{s-1} R_{l,s}~ \b_{r,l} = 0.\] 
By construction  we have $\deg(\b_{r,r}) > \deg(\b_{l,r})$ if $l >r$.
 So we get again that if $R_{r,s} \neq 0$ then  $\deg(R_{r,s}) < M = 0$. This means we have $R_{r,j} =  0$ for all $j \le  i_1-1$, $j\neq r$.
\qed

\subsection{Proof of 3}
We will no longer assume in the proof of the surjectivity that $I_0$ is a lex segment ideal.

We want to find for every ideal $I \subset K[x,y]$ such that $\ini(I) = I_0$, a Hilbert-Burch matrix of the form $X+A$ with $A \in \AA{I_0}$.
It is easy to see that we can find a \GB~  $\{f_0,\ldots, f_t\}$ for $I$  with $\ini(f_i) = x^{t-i}y^{m_i}$ and leading coefficient 1. Due to the form of the leading terms of these polynomials, we can also assume that the monomials in the support of the $f_i$'s are not divisible by $x^t$ (except for the leading term of $f_0$). Otherwise, if  there exists an $i$ such that $cx^{t+h}y^l \in \textrm{Supp}(f_i)$, for some $h,l \ge 0$ and $c \in K^{*}$, we modify  $f_i$ to be $f_i - cx^hy^lf_0$. 

The S-polynomials $y^{d_i}f_{i-1} - xf_i$ have no term in their support divisible by $x^{t+1}$. So their reduction to 0 will be of the following form:
\begin{equation}\label{syz}
y^{d_i}f_{i-1} - xf_i + \sum_{j=0}^{t} a_{j,i} f_j = 0,
\end{equation}
with  $a_{i,j} \in K[y],~\forall~ i,j$ and
$ \ini(a_{j,i} f_j) \le \ini(y^{d_i } f_{i-1} - xf_{i}), \textrm{for all}~ j = 0,\ldots,t.$
The fact that $a_{i,j} \in K[y]$ follows by slightly modifying the proof of Lemma~\ref{lema1}.

These S-polynomials correspond to syzygies of the leading terms of the $f_i$'s:
\[y^{d_i}(x^{t-i+1}y^{m_{i-1}}) - x(x^{t-i}y^{m_i}).\]
As these syzygies generate the syzygy module of  $\ini(f_i)$, Schreyer's theorem implies that the equations (\ref{syz}) generate the syzygy module of the $f_i$'s.

Setting these syzygies as columns of a matrix, we obtain a $(t+1) \times t$ matrix of the form $X + A$, where the entries of $A$ are elements of $K[y]$. By the Hilbert-Burch theorem we have that the $t$-minors of this matrix generate the ideal $I$.

By the inequality of  the leading terms in (\ref{syz}) we obtain the following restrictions on the degrees of the $a_{i,j}$: 
\begin{equation}\label{grade}
 \deg(a_{i,j}) \le \left.\bigg\{ \begin{array}{cc}
                                 i - j  +m_j - m_{i-1} -1&\textup{if}~ i \le j, \\
                               \rule{0pt}{3ex} i - j  +m_j - m_{i-1} \phantom{-1..}&\textup{if} ~ i > j.\\
                                   \end{array}\right.
\end{equation}

Now we will show how to modify this matrix in order to obtain a new matrix $X+A'$ with $A' \in \AA{I_0}$. It is easy to see that elementary operations on the Hilbert-Burch matrix  do not change the fact that the maximal minors generate the ideal.
We will use a sequence of  pairs of standard operations, that we will call reduction moves.

Take $i \neq j$, with $i \in \{1,\ldots, t+1\}$ and $j \in \{1,\ldots,t\}$. Suppose we have
\begin{equation}
 \deg(a_{i,j}) \ge \left.\bigg\{ \begin{array}{cc}
                                d_i&\textup{if}~ i < j, \\

      \rule{0pt}{3ex}d_j&\textup{if} ~ i > j.\\
                                   \end{array}\right.
\end{equation}
If $i < j$, (resp. $i>j$)  denote  by $q_{i,j}$ the quotient of the division of $a_{i,j}$ by $y^{d_i} + a_{i,i}$, (resp. $y^{d_j} + a_{j,j}$). So we  have:
\begin{equation}
 a_{i,j} = \left.\bigg\{ \begin{array}{cc}
                                (y^{d_i} + a_{i,i})q_{i,j} + r_{i,j}, ~\textrm{with}~ \deg(r_{i,j}) < d_i,&\textup{if}~ i < j, \\

      \rule{0pt}{3ex}( y^{d_j} + a_{j,j})q_{i,j} + r_{i,j}, ~\textrm{with}~ \deg(r_{i,j}) < d_j,&\textup{if} ~ i > j.\\
                                   \end{array}\right.
\end{equation}
Notice that, as the degree of $a_{i,j}$ is bounded as in (\ref{grade}), we also have: 
\begin{equation}\label{q}
\deg(q_{i,j}) \le \left.\bigg\{ \begin{array}{cc}
                                 i-j +m_j - m_i -1&\textup{if}~ i < j, \\

        \rule{0pt}{3ex}i - j  +m_{j-1} - m_{i-1}&\textup{if} ~ i > j.\\
                                   \end{array}\right.
\end{equation}
We will call  a ($i,j$)-reduction move the sequence of the following two standard operations:\\
\emph{If $i<j$}\begin{enumerate}
\item[-] Add the $i$th column multiplied by $ - q_{i,j}$ to the $j$th column. 
\item[-] Add the ($j+1$)th row multiplied by $ q_{i,j}$ to the ($i+1$)th row.
\end{enumerate}
\emph{If $i>j$}\begin{itemize}
\item[-] Add the $j$th row multiplied by $ - q_{i,j}$ to the $i$th row. 
\item[-] If $j \ge 2$, add the ($i-1$)th column multiplied by $ q_{i,j}$ to the $(j-1)$th column.
\end{itemize}
In the second case, when $j=1$ we only do the first move.

The first operation, reduces the degree of the entry in the position ($i,j$), by replacing $a_{i,j}$ with $r_{i,j}$. The second one cancels the multiple of $x$ that appeared in the position ($i+1,j$) if $i<j$, (respectively the position ($i,j-1$) if $i>j$) as a consequence of the first move. 
Let us see that after each such reduction move, the degrees in the new matrix are still bounded as in (\ref{grade}). We take a look at what happens for $i<j$, the other case being similar.

For the first operation, for all $k = 1,\ldots, t+1$, we have:
\[\deg(a_{k,i}q_{i,j})\le (k-i+m_i-m_{k-1}) + (i-j +m_j - m_i -1) = k  - j  +m_j -m_{k-1} -1.\]

For the second operation, for all $k= 1, \ldots, t$, we have:
\[\deg(a_{j+1,k}q_{i,j})\le j+1 - k + m_k - m_j + i-j +m_j - m_i -1 = i+1  - k  +m_k -m_i  -1.\]

As it is clear that every reduction move influences more elements, not just the one it is aimed at, we will have to determine which entries are influenced \lq\lq most\rq\rq. This way, we will be able to conclude that after a finite sequence of reduction moves we can  reduce the degree of an entry by 1 and leave all other degrees unchanged. Thus, in the end we will be able to reduce the matrix to the desired form. Notice that, once the matrix is $X+A'$ with   $A' \in \AA{I_0}$, by definition we cannot make any more reduction moves.\\

Let us denote with Red$_{i,j}$ the reduction moves. We will say that  Red$_{i,j}$ is maximal in $a_{k,l}$ (or just in $(k,l)$) if $a_{k,l}$ is modified such that $\deg(a_{k,l})$ reaches the upper bound given in (\ref{grade}). It is easy to see that in order to get this, also $\deg(q_{i,j})$ has to reach the upper bound given in (\ref{q}).

In the next part, using easy computations, we will find the indices ($k,l$) in which Red$_{i,j}$ is maximal. There are two main cases depending on $i$ and $j$, each of them having four sub-cases. The computations follow in each case the same pattern. As they are trivial but rather long, we will present  the details only in the first two sub-cases.

\emph{Case 1:} $i < j$. We have to have $\deg(q_{i,j}) = i - j + m_j - m_i -1$ according to (\ref{q}). By definition Red$_{i,j}$ will act on the elements of the $j$th column and on those of the ($i+1$)th row. Let us first take a look at what happens on the $j$th column.

Let $k \in \{1,\ldots,t+1\}$. We want to see what the degree of $a_{k,i}q_{i,j}$ could be:

If $k<i$, then
\begin{eqnarray*}
\deg(a_{k,i}q_{i,j}) & = & k- i + m_i - m_{k-1} -1 + i - j + m_j - m_i -1\\
\rule{0pt}{3ex}~     & =& (k  - j + m_j - m_{k-1} -1) -1,
\end{eqnarray*}
so it cannot reach the upper bound in (\ref{grade}).

If $k \ge i$, then
\begin{eqnarray*}
\deg(a_{k,i}q_{i,j}) & = & k- i + m_i - m_{k-1}  + i - j + m_j - m_i -1\\
\rule{0pt}{3ex}~     & =& k  - j + m_j - m_{k-1} -1,
\end{eqnarray*}
so it can be maximal only if $k < j$.

On the ($i+1$)th row, with similar computations we obtain that.\\
If $k\le j+1$ the degree of $a_{j+1,k}q_{i,j}$ reaches the upper bound only if $k > i+1$.\\
If $k > j+1$ the degree of $a_{j+1,k}q_{i,j}$ cannot be maximal.\\
So for the reduction moves that act above the diagonal the positions that could be maximal are:
\begin{eqnarray*}
(k,j) & &\textrm{if}~~ i<k<j,\\
(i+1,k) &&\textrm{if}~~ i+1 < k \le j+1.
\end{eqnarray*}

\emph{Case 2:} $i > j$. We have to have $\deg(q_{i,j}) = i - j + m_{j-1} - m_{i -1}$ according to (\ref{q}). By definition Red$_{i,j}$ will act on the elements of the $i$th row and on those of the ($j-1$)th column. Using arguments similar to the ones above we obtain that for the reduction moves that act below the diagonal the positions that could be maximal are:
\begin{eqnarray*}
(i,k) & &\textrm{if}~~ k<j \phantom{-1}~\,~~\textrm{or}~~ k>i,\\
(k,j-1) & & \textrm{if}~~  k  < j-1~~ \textrm{or}~~ k \ge i-1.
\end{eqnarray*}
Here is a graphical representation of the positions that may be maximal for Red$_{i,j}$:\\
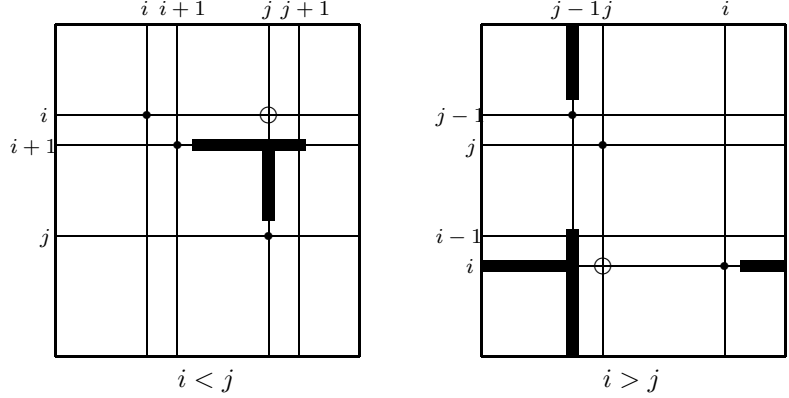
\begin{center}
\setlength{\unitlength}{4mm}
\begin{picture}(25,14)
\linethickness{0.2mm}
\multiput(1,2)(10,0){2}{\line(0,1){11}}
\multiput(15,2)(10,0){2}{\line(0,1){11}}

\multiput(1,2)(0,11){2}{\line(1,0){10}}
\multiput(15,2)(0,11){2}{\line(1,0){10}}
\linethickness{0.1mm}
\put(1,6){\line(1,0){10}}
\multiput(1,9)(0,1){2}{\line(1,0){10}}

\multiput(15,5)(0,1){2}{\line(1,0){10}}
\multiput(15,9)(0,1){2}{\line(1,0){10}}
\multiput(4,2)(1,0){2}{\line(0,1){11}}
\multiput(8,2)(1,0){2}{\line(0,1){11}}

\multiput(18,2)(1,0){2}{\line(0,1){11}}
\put(23,2){\line(0,1){11}}

\multiput(4,10)(1,-1){2}{\circle*{0.3}}
\multiput(18,10)(1,-1){2}{\circle*{0.3}}
\put(8,6){\circle*{0.3}}
\put(23,5){\circle*{0.3}}

\put(8,10){\circle{0.5}}
\put(19,5){\circle{0.5}}
\linethickness{1.5mm}
\put(5.5,9){\line(1,0){3.7}}
\put(15,5){\line(1,0){3}}
\put(23.5,5){\line(1,0){1.5}}
\put(8,6.5){\line(0,1){2.5}}
\put(18,2){\line(0,1){4.2}}
\put(18,10.5){\line(0,1){2.5}}

{\footnotesize
\put(0.5,5.75){$j$}
\put(-0.5,8.75){$i+1$}
\put(0.5,9.75){$i$}
\put(14.5,4.75){$i$}
\put(13.5,5.75){$i-1$}
\put(14.5,8.75){$j$}
\put(13.5,9.75){$j-1$}
\put(3.8,13.33){$i$}
\put(4.4,13.33){$i+1$}
\put(7.8,13.33){$j$}
\put(8.4,13.33){$j+1$}
\put(17.3,13.33){$j-1$}
\put(19,13.33){$j$}
\put(22.9,13.33){$i$}
}
\put(5,1){$i<j$}
\put(19,1){$i>j$}

\end{picture}
\end{center}
The circle represents the position of the $a_{i,j}$ that is being reduced, the dots represent entries on the diagonal. The thin lines are columns, respectively rows, and the thick lines represent the positions in which maximal elements for Red$_{i,j}$ may appear.

Now we will show how, using these reduction moves, we can bring the Hilbert-Burch matrix to the form we want to.
We will proceed by induction on $t$.
When $t =1$ there is not much to prove, so we can assume by induction that the upper left $t \times (t-1)$ part of the matrix is already in the form we want.  
We will show now how we can bring the elements of the last row and column to the desired form. 
We will start with the last row. 

 Suppose also that we have $\deg(a_{t+1,j}) = t+1-j+m_j-m_t > d_j-1$ and that we have already brought the elements $a_{t+1,t} , \ldots, a_{t+1,j+1}$ to the desired degree for some $j \in \{1,\ldots ,t \}$.  

First we do the reduction move Red$_{t+1,j}$. This will have maximal degree. Then we will apply the other reduction moves that are necessary to bring the $t \times (t-1)$ upper left part to the desired form. This can be done by induction. It is easy to see from the graphical representation, that for all these moves, the elements $a_{t+1,j},\ldots,a_{t+1,t}$ will not be maximal. Now, also by induction we will bring to the desired form also the elements $a_{t+1,t},\ldots,a_{t+1,j+1}$. Again, as the reduction moves will not be of maximal degree, by definition the element $a_{t+1,j}$ will not be maximal for any of them. So after performing all these reductions we will have $\deg(a_{t+1,j}) < t+1-i+m_j-m_t$.

This whole sequence of operations depends on the first reduction move Red$_{t+1,j}$. It is easy to notice that, even if we will start with a reduction that is not of maximal degree, we will still reduce the degree of $a_{t+1,j}$ by at least one. So we can do this until $\deg(a_{t+1,j}) \le d_j-1$.

Let us now bring also the elements on the last column to the desired form. Suppose that the first $t-1$ columns and $a_{t+1,t}$ are of the desired form. Let  $\deg(a_{i,t}) = i-t-1+m_t - m_{i-1} > d_i-1$ and  suppose that we brought  $a_{1,t},\ldots, a_{i-1,t}$ to the desired form. 

We apply now Red$_{i,t}$ which will be of maximal degree. Then we will  bring the rest of the matrix, that we assumed had already  the desired form, in the desired form again. These operations can be done by induction, and it is easy to see that the elements $a_{1,t}, \ldots,a_{i-1,t}$ will not be maximal. So also $a_{i,t}$ will not be maximal for any reduction. This means that we have reduced its degree by at least one.

This whole sequence of operations depends on the first reduction move Red$_{i,t}$ and, as we explained in the previous case, even if we will start with a reduction that is not of maximal degree, we will still reduce the degree of $a_{i,t}$ by at least one. So we can do this until $\deg(a_{i,t}) \le d_i-1$. We have thus proven the surjectivity.

\newpage

\section{Dimension}
Let $I_0$ be a monomial lex-segment ideal of $K[x,y]$ as in the previous section. In this part we will show how to compute the dimension of the affine space $V(I_0)$ that we parametrized.  For every $i\ge 0$ we denote by $h_i= \dim_K((R/I_0)_i)$, that is the value of the \HF~ of $R/I_0$ in $i$. Using the notation introduced so far we have:
$$h_i = i+1, \textup{~for~} 0\le i \le t-1.$$
If we denote by $\b_{0,i} = \b_{0,i}(I_0)$ the number of minimal generators of $I_0$ of degree $i$ we also have the following:
\begin{eqnarray*}
\b_{0,t}& = &h_{t-1}- h_t+1,\\
\b_{0,i}&=& h_{i-1}-h_{i}, \textup{~for~} i>t,\\
h_i&=& \sum_{j>i} \b_{0,j}, \textup{~for~} i\ge t .
\end{eqnarray*}
Recall that
$ \mathcal{J} $= $\{j_1,j_2,\ldots, j_p\}$= \mbox{$\{ j \in 1, \ldots , t ~:~ d_j \ge 2 \}.$}
If we set by convention $j_{p+1} = t+1$,  we have $j_1= b_{0,t}$ and $j_{i+1} -j_i = \b_{0,t-j_i+m_{j_i}}$.  Note that these equalities depend on the fact that $I_0$ is a lex-segment ideal.

\begin{proposition}\label{dim}
Let $I_0 \subset R$ be a monomial lex-segment ideal. Using the above notation we have the following formula:
\[\dim(V(I_0)) = \dim_K(R/I_0) + 1 + \sum_{i\ge 1} h_i(h_{i-1}-h_{i-2}).\]
\end{proposition}
\begin{proof}
To prove the proposition we just have to count the number of coefficients that appear in a matrix $A \in \AA{I_0}$.  As a polynomial in $K[y]$ of degree at most $r$ has $r+1$ coefficients, form the entries on and above the diagonal we get:
\[\sum_{i=1}^{t} (t-i+1)d_i  = \sum_{i=1}^{t} m_i = \dim_K(R/I_0).\]
We will count the number of coefficients below the diagonal in the following way: 
\[ \sharp \textup{~of entries} + \sharp \textup{~of entries of degree} 1 - \sharp \textup{~of 0s}.\]
The number of entries in a  triangular block of size $t$ is $t(t+1)/2$. For $i<t$ we have
$h_i(h_{i-1}-h_{i-2}) = i+1$, so the number of entries is
\[\frac{t(t+1)}{2} = 1 + \sum_{i=1}^{t-1}h_i(h_{i-1}-h_{i-2}).\]
By looking at the shape of the matrices in $\AA{I_0}$ described in Figure \ref{allA} and \ref{partA}, it is easy to see that the number of entries below the diagonal that have degree 1 is
\[ t+1-j_1 = \sum_{i>t}\b_{0,i} = h_t.\]
The entries of a matrix in $\AA{I_0}$ that are always $0$ are grouped in vertical blocks of size \mbox{$(j_{p+1}-j_{i+2})\times (j_{i+1}-j_{i})$}. As we have
\[ j_{p+1}-j_i = \sum_{k >i-1} \b_{0,t-j_k+m_{j_k}} = h_{(t-j_i+m_{j_i})+1},\]
and 
$ j_{i+1}-j_{i} = \b_{0,t-j_{i}+m_{j_{i}}} = h_{ (t-j_{i}+m_{j_{i}}) - 1} - h_{t-j_{i}+m_{j_{i}}},$
we obtain that the number of zeros is 
\[ \sum_{i>t}h_i(h_{i-1}-h_{i}).\]
Taking into account that $h_t = h_t(h_{t-1}-h_{t-2})$ and adding the above numbers with their proper signs, we obtain that below the diagonal we have exactly
\[1 + \sum_{i\ge1} h_i(h_{i-1}-h_{i-2})\]
coefficients that are parameters.
\end{proof}

As we have $\dim_K(R/I_0) = \sum_{i\ge0} h_i$, we can write the formula of the dimension in a more compact way, namely:
\[ \dim(V(I_0)) = 1 + \sum_{i\ge 0 }h_i(h_{i-1}-h_{i-2} +1).\]

We will see in Section 7 that, if the characteristic of  $K$ is   $p> \max\{j~:~h_j\neq0\}$ or zero, then  $\AA{I_0}$ also parametrizes an affine open subset of the Hilbert function strata of $Hilb^{n}(\mathbb{P}^2)$. This means that the formula found in Proposition \ref{dim} is also valid for $Hilb^H(\mathbb{P}^2)$, where $H$ denotes the Hilbert function of an ideal $ K[x,y,z]/I $, where $I$ defines  a zero-dimensional scheme in $\mathbb{P}^2$. With this notation the $h$'s in the dimension formula are: $h_i = H_i - H_{i-1}$.  The dimension of $Hilb^H(\mathbb{P}^2)$ was determined by G.Gotzmann in \cite{got}. A different formula was also given by G. Ellingsrud and S. A. Str\o mme in \cite{ES1,ES2}, then by A. Iarrobino and V. Kanev in \cite{IK}. The latest formula was found by K. De Naeghel and M. Van den Bergh in \cite{NvdB}. Our formula is nearest to the latter one. Although the methods used in the two proofs are different, a short computational passage transforms one formula  into the other. 

Given that we compute the dimension by counting the parameters in the Hilbert-Burch matrix, it is easy to check that the following bounds hold. The lower bound generalizes slightly the result obtained by  K. De Naeghel and M. Van den Bergh in \cite[Corollary 6.2.3]{NvdB}.
\begin{corollary} Let $K$ be a field of characteristic  $p> \max\{j~:~h_j\neq0\}$ or of characteristic 0 and
let $H$ and $h$ be as above. Denote by $\Lex(h)$ the lex-segment ideal of $R= K[x,y]$ with Hilbert function $h$, by $n = \dim_K(R/\Lex(h))$ and by $t = \mu(\Lex(h))-1$ the number of  minimal generators of $\Lex(h)$ minus 1. For $n \ge 2 $ we have 
\[ \max\{n+t, ~n+2\}~\le~  \dim Hilb^H(\mathbb{P}^2) ~\le~ 2n.\]
\end{corollary}
\begin{proof}
We have seen in the proof of Proposition \ref{dim} that the number of parameters on and above the diagonal is always $n$. By the bounds in (\ref{A}) we obtain that below the diagonal we may have also at most $n$ parameters. 

On the other hand, as we are considering matrices in $\AA{\Lex(h)}$, all the $d_i$'s are greater or equal to 1. Notice that $t$ has the same meaning as in the proof of Theorem \ref{main}. So for $i = 1,\ldots, t$ each entry indexed $(i+1, i)$ contributes with at least one parameter. This proves the lower bound whenever $t \ge 2$. In the extreme case when $\Lex(h)$ is generated by two elements, as $n\ge2$, we must have $d_1\ge 2$. So in this case there are exactly two parameters that appear below the diagonal.
\end{proof}

\section{Examples}

We will show now with three examples how the proof of the main theorem  works. We  start with a \lq\lq small\rq\rq~example from which it will be easier to see the main idea behind the proof of the injectivity. Then, we are forced to choose a rather \lq\lq large\rq\rq~example in order to present  the more technical arguments that we use in the proof. The last example shows how to find the canonical Hilbert-Burch matrix for a given ideal $I$, i.e. the corresponding matrix $A \in \AA{\ini(I)}$. 
\subsection{Example 1}
Let $I_0$ be the following ideal:
\[I_0 = (x^3, x^2y^5, xy^7, y^{11}).\] 
So we have: $m_0 =  0$, $m_1 = 5 $, $m_2 = 7 $, $m_3 = 11$ and $d_1 = 5,$ $d_2 = 2$, $d_3 = 4$.
The sets of "special" indices are: $ \mathcal{I}  = \{1,3\} $ and $ \mathcal{J}  =  \{1,2,3\}$.
The matrix that bounds the degrees of the entries of a matrix $A \in \AA{I_0}$ is
\[\left( \begin{array}{rrr}
  4 & 4 & 4\p\\
  1 & 1 & 1\p\\
  0 & 1 & 3\p\\
 -3 & -2 & \p1\p
\end{array}\right).\]
Note that the Hilbert function of $R/I_0$ is $h_{R/I_0} = (1, 2, 3, 3, 3, 3, 3, 2, 1, 1, 1, 0).$
Now let $A$ and $B$ be two matrices in $\AA{I_0}$. The matrix $X+A$ will be
\[X+A = \left( \begin{array}{rrr}
  y^5 + a_{1,1}  & a_{1,2} & a_{1,3}\\
  -x + a_{2,1} & y^2 + a_{2,2} & a_{2,3}\\
  a_{3,1} & -x + a_{3,2} & y^4 + a_{3,3}\\
 0 & 0 & -x+a_{4,3}
\end{array}\right).\]
The matrix $X+B$ will have a similar form. Using the same notations as in the proof  we can write:
\[\begin{array}{rclcrcl}
 f_0 &= &g_0&\phantom{thismuchspace}& f_1 &= &g_1\\
 f_2 &=& g_2 + R_{1,2}~g_1&\phantom{thismuchspace}&f_3 &=& g_3 + R_{1,3}~g_1 + R_{2,3}~g_2
\end{array}\]  
The transition matrix $R$ will have actually two blocks (even if $q = \sharp\mathcal{I}=2$). This is because the first block is degenerated ($i_0=i_1=1$). The first column is not considered  part of any block.
\[ R = \left(
\begin{array}{l|ll|l}
1&0&0&0\\
0&1&R_{1,2}&R_{1,3}\\
0&0&1&R_{2,3}\\
0\phantom{_{2,3}}&0\phantom{_{2,3}}&0&1
\end{array}
\right).
\]  
It is easy to see that,  as the columns of $X+A$ are syzygies for $(f_0, f_1, f_2, f_3)$, the columns of $R(X+A)$ will be syzygies   for $(g_0, g_1, g_2, g_3)$. From these we will subtract the necessary multiples of the columns of $B$ in order to obtain syzygies with entries in $K[y]$:
\[  {\scriptsize E_1 = \left(
\begin{array}{r}
 y^5 + a_{1,1} \\ 
  -x + a_{2,1} + a_{3,1}R_{1,2}\\ 
  a_{3,1}                                      \\ 
 0\phantom{_{3,1}}                                                    
\end{array}
\right) - 
\left(
\begin{array}{r}
 y^5 + b_{1,1} \\ 
  -x + b_{2,1} \\ 
  b_{3,1}                                      \\ 
 0\phantom{_{3,1}}                                                    
\end{array}
\right),
}\] 

\[{\scriptsize  E_2 = \left(
\begin{array}{r}
 a_{1,2}                                                     \\  
  y^2 + a_{2,2} + (-x +a_{3,2})R_{1,2}   \\   
 -x + a_{3,2}                                                 \\  
  0\phantom{_{3,1}}                                                                    

\end{array}
\right) - 
\left(
\begin{array}{r}
 b_{1,2}                                                     \\  
  y^2 + b_{2,2}  \\   
 -x + b_{3,2}                                                 \\  
  0\phantom{_{3,1}}                                                                    

\end{array}
\right) -
\left(
\begin{array}{r}
( y^5 + b_{1,1})R_{1,2} \\ 
(  -x + b_{2,1} )R_{1,2}\\ 
  b_{3,1}\phantom{)}R_{1,2}                                      \\ 
 0\phantom{)R_{2,3}}                                                    

\end{array}
\right),
}\]  

  \[{\scriptsize E_3 = \left(
\begin{array}{r}
a_{1,3}\\
 a_{2,3} + (y^4 +a_{3,3})R_{1,2} + (-x + a_{4,3})R_{1,3}\\
y^4 + a_{3,3} + (-x + a_{4,3})R_{2,3}\\
-x+a_{4,3}
\end{array}
  \right)-
  \left(
\begin{array}{r}
 (y^5 + b_{1,1})R_{1,3} \\ 
 ( -x + b_{2,1})R_{1,3} \\
  b_{3,1} \phantom{)} R_{1,3}                                     \\ 
 0 \phantom{_{3,1}}\phantom{)R_{2,3}}                                                   
\end{array}
\right)-
  }\]
  \[{\scriptsize -\left(
\begin{array}{r}
 b_{1,2} \phantom{)}R_{2,3}                                                    \\  
 ( y^2 + b_{2,2})R_{2,3}  \\   
( -x + b_{3,2})R_{2,3}                                                 \\  
  0\phantom{_{3,1}}\phantom{)R_{2,3}}                                                                    
\end{array}
  \right)-
\left(
\begin{array}{r}
b_{1,3}\\
 b_{2,3} \\
y^4 + b_{3,3} \\
-x+b_{4,3}
\end{array}
  \right).
  }\]
  
  From $E_1$ we get that
  $E_{2,1}= a_{2,1} + a_{3,1}R_{1,2} - b_{1,2} = 0,$
  but we cannot draw any conclusion from here, as $a_{3,1}$ may also be 0.
  
  From the first entry of $E_2$ we have
  \[E_{1,2}=a_{1,2} - b_{1,2} - ( y^5 + b_{1,1})R_{1,2} = 0.\]
  As $\deg(a_{1,2}) \le 4$ and $\deg(b_{1,2}) \le 4$ we obtain that $R_{1,2} = 0$. We set $R_{1,2} = 0$ in $E_3$ and we get
  \[
  \begin{array}{rcl}
E_{1,3}=  a_{1,3} -b_{1,3} - (y^5 + b_{1,1})R_{1,3} -b_{1,2} R_{2,3}  &=&0,\\
  \rule{0pt}{3ex} 
E_{2,3}=  a_{2,3}  +  a_{4,3}R_{1,3} -b_{2,1}R_{1,3} - ( y^2 + b_{2,2})R_{2,3} - b_{2,3} &=& 0.
  \end{array}
  \]
  From $E_{1,3}$, as all the $a$'s and $b$'s  have degree less then 4 we get that if $R_{1,3} \neq 0$ then
  \[\deg(R_{1,3}) < \deg(R_{2,3}).\]
   From the second equation, as this time all the $a$'s and $b$'s  have degree less then 1 we get that if $R_{2,3} \neq 0$ then
   \[\deg(R_{2,3}) < \deg(R_{1,3}).\]
   This means that we actually  must have $R_{1,3} = R_{2,3} = 0$.

  \subsection{Example 2}
In the previous example, we did not have to change the equations. Also, as all indices were \lq\lq special\rq\rq~in that case, we obtained directly strict inequalities. In the next example we will see all the possible types of situations that may arise.
Let $I_0$ be the following ideal:
\[\begin{array}{ccl}
I_0 &= &(x^{12},  x^{11}y^3, x^{10}y^4, x^9y^5, x^8y^{10}, x^7y^{11},x^6y^{12},\\
\rule{0pt}{3ex} && \phantom{(} x^5y^{14}, x^4y^{15}, x^3y^{16}, x^2y^{19}, xy^{20}, y^{21}).
\end{array}\]
So $t = 12$, the $m$'s and the $d$'s are:
\[\begin{array}{ccl}
m&=&(0,3,4,5,10,11,12,14,15,16,19,20,21),\\
\rule{0pt}{3ex} d&=&(3,1,1,5,1,1,2,1,1,3,1,1).
\end{array}
\]
The sets of "special" indices are:
\[\begin{array}{ccl}
 \mathcal{I} & = &\{1,4,10\}, \\
 \mathcal{J} & = & \{1,4,7,10\}.
 \end{array}\]
 
We denote by  $h = (h_i)_{i\ge0}$ the $h$-vector of $R/I_0$. So  $h_i$ is actually the value of the \HF~ of $R/I_0$ in $i$.  We have 
\[ h = ( 1, 2, 3, 4, 5, 6, 7, 8, 9, 10, 11, 12, 12, 12, 9, 9, 9, 9, 6, 3, 3) \]
In the picture below it is easy to notice that the number of entries below the diagonal of degree one is $h_{12} =12$. Also, notice that the three blocks of zeros have sizes $h_{15}(h_{13}-h_{14}) =9(12-9) =27$,
$h_{19}(h_{17}-h_{18}) =3(9-6) =9$ and $h_{20}(h_{18}-h_{19}) =3(6-3) =9$.
 One can also easily check that the number of parameters, namely 195,
   is equal to $1+\sum_{i\ge0} h_i(h_{i-1}-h_{i-2}+1)$.
We will now look at a general matrix $A \in \AA{I_0}$. This will be an $13 \times 12$ matrix, with entries polynomials in $K[y]$. 
In order to emphasize the maximal possible degree of each $a_{i,j} \neq 0$ we  denote as in Subsection 3.2:
\[ a_{i,j} = \bigg\{ 
\begin{array}{ll}
 {\scriptstyle\bullet}~y^{m_{i,j}} & ,\textup{if}~ m_{i,j} > 0, \\
 c                                                     & ,\textup{if}~ m_{i,j} = 0.
  \end{array} \]
With this notation the matrix $A$ is:
\[
 \left(
\begin{array}{rrr|rrr|rrr|rrr}
 {\scriptstyle\bullet}~y^2 &  {\scriptstyle\bullet}~y^2 &  {\scriptstyle\bullet}~y^2 &  {\scriptstyle\bullet}~y^2 &  {\scriptstyle\bullet}~y^2 &  {\scriptstyle\bullet}~y^2 &  {\scriptstyle\bullet}~y^2 &  {\scriptstyle\bullet}~y^2 &  {\scriptstyle\bullet}~y^2 &  {\scriptstyle\bullet}~y^2 &  {\scriptstyle\bullet}~y^2 &  {\scriptstyle\bullet}~y^2\\
 \hline
  {\scriptstyle\bullet}~y\pp &c\pp & c\pp & c\pp & c\pp & c\pp & c\pp & c\pp & c\pp & c\pp & c\pp & c\pp\\
  {\scriptstyle\bullet}~y\pp & c\pp & c\pp  & c\pp & c\pp & c\pp & c\pp & c\pp & c\pp & c\pp & c\pp & c\pp\\
 {\scriptstyle\bullet}~y\pp & c\pp & c\pp &  {\scriptstyle\bullet}~y^4 & {\scriptstyle\bullet}~y^4 & {\scriptstyle\bullet}~y^4 & {\scriptstyle\bullet}~y^4 & {\scriptstyle\bullet}~y^4 & {\scriptstyle\bullet}~y^4 & {\scriptstyle\bullet}~y^4 & {\scriptstyle\bullet}~y^4 & {\scriptstyle\bullet}~y^4\\
  \hline
 0\pp & 0\pp & 0\pp &  {\scriptstyle\bullet}~y\pp & c\pp  & c\pp & c\pp & c\pp & c\pp & c\pp & c\pp & c\pp\\
 0\pp & 0\pp & 0\pp & {\scriptstyle\bullet}~y\pp &  c\pp & c\pp & c\pp & c\pp & c\pp & c\pp & c\pp & c\pp\\
 0\pp & 0\pp & 0\pp & {\scriptstyle\bullet}~y\pp & c\pp & c\pp &  {\scriptstyle\bullet}~y\pp & {\scriptstyle\bullet}~y\pp & {\scriptstyle\bullet}~y\pp & {\scriptstyle\bullet}~y\pp & {\scriptstyle\bullet}~y\pp & {\scriptstyle\bullet}~y\pp\\
  \hline
 0\pp & 0\pp & 0\pp & c\pp & c\pp & c\pp & {\scriptstyle\bullet}~y\pp & c\pp & c\pp & c\pp & c\pp & c\pp\\
 0\pp &0\pp & 0\pp & c\pp & c\pp & c\pp & {\scriptstyle\bullet}~y\pp & c\pp & c\pp & c\pp & c\pp & c\pp\\
 0\pp &0\pp &0\pp & c\pp & c\pp & c\pp & {\scriptstyle\bullet}~y\pp & c\pp & c\pp  &  {\scriptstyle\bullet}~y^2 & {\scriptstyle\bullet}~y^2 & {\scriptstyle\bullet}~y^2\\
  \hline
 0\pp &0\pp &0\pp & 0\pp&0\pp &0\pp & 0\pp & 0\pp &\pp 0\pp &  {\scriptstyle\bullet}~y\pp & c\pp & c\pp\\
 0\pp &0\pp &0\pp & 0\pp&0\pp &0\pp & 0\pp & 0\pp &\pp 0\pp &  {\scriptstyle\bullet}~y\pp & c\pp & c\pp\\
 0\pp &0\pp &0\pp & 0\pp&0\pp &0\pp & 0\pp & 0\pp &\pp 0\pp &  {\scriptstyle\bullet}~y\pp & c\pp & c\pp\\
  \end{array} \right). \]

Let $B \in \AA{I_0}$ be another matrix as in the proof of the injectivity. Suppose that they both parametrize the same ideal. The transition matrix $R$ from $X+A$ to $X+B$ is:
\[
\begin{array}{|l|lll|llllll|lll|}
&i_1&&&i_2&&&&&&i_3&&\\
\hline
1 & 0 & 0 & 0 & 0 & 0 & 0 & 0 & 0 & 0 & 0 & 0 & 0 \\
0 & 1 & 0 & 0 & R_{1,4} & R_{1,5} & R_{1,6} & R_{1,7} & R_{1,8} & R_{1,9} & R_{1,10} & R_{1,11} & R_{1,12} \\
0 & 0 & 1 & 0 & R_{2,4} & R_{2,5} & R_{2,6} & R_{2,7} & R_{2,8} & R_{2,9} & R_{2,10} & R_{2,11} & R_{2,12} \\
0 & 0 & 0 & 1 & R_{3,4} & R_{3,5} & R_{3,6} & R_{3,7} & R_{3,8} & R_{3,9} & R_{3,10} & R_{3,11} & R_{3,12} \\
0 & 0 & 0 & 0 & 1 & 0 & 0 & R_{4,7} & R_{4,8} & R_{4,9} & R_{4,10} & R_{4,11} & R_{4,12} \\
0 & 0 & 0 & 0 & 0 & 1 & 0 & R_{5,7} & R_{5,8} & R_{5,9} & R_{5,10} & R_{5,11} & R_{5,12} \\
0 & 0 & 0 & 0 & 0 & 0 & 1 & R_{6,7} & R_{6,8} & R_{6,9} & R_{6,10} & R_{6,11} & R_{6,12} \\
0 & 0 & 0 & 0 & 0 & 0 & 0 & 1 & 0 & 0 & R_{7,10} & R_{7,11} & R_{7,12} \\
0 & 0 & 0 & 0 & 0 & 0 & 0 & 0 & 1 & 0 & R_{8,10} & R_{8,11} & R_{8,12} \\
0 & 0 & 0 & 0 & 0 & 0 & 0 & 0 & 0 & 1 & R_{9,10} & R_{9,11} & R_{9,12} \\
0 & 0 & 0 & 0 & 0 & 0 & 0 & 0 & 0 & 0 & 1 & 0 & 0 \\
0 & 0 & 0 & 0 & 0 & 0 & 0 & 0 & 0 & 0 & 0 & 1 & 0 \\
0 & 0 & 0 & 0 & 0 & 0 & 0 & 0 & 0 & 0 & 0 & 0 & 1 \\
\hline
\end{array}
\]
Notice that again the first block is degenerated, so we only have 3 blocks. As $i_2=j_2$ and $j_1=1$ we have that there is nothing to prove for the second block either.   
It is clear that when multiplying with the first three columns of $X+A$, the $R_{k,l}$'s do not appear. Also when multiplying columns 4 to 9, the last three columns of the matrix $R$ do not play any role. In this example we will look just at the $R_{i,j}$'s with $j \le 9$, namely the third block. The order that we introduced in the proof is in this case the following:
\[
\begin{array}{ccccccccccccccc}
&R_{1,4}&<&R_{1,5}&<&R_{1,6}&<&R_{2,4}&<&R_{2,5}&<&R_{2,6}&<&R_{3,4}&<\\
<&R_{3,5}&<&R_{3,6}&<&R_{1,7}&<&R_{1,8}&<&R_{1,9}&<&R_{2,7}&<&R_{2,8}&<\\
<&R_{2,9}&<&R_{3,7}&<&R_{3,8}&<&R_{3,9}&< &\ldots&<& R_{6,9}&&&\\
\end{array}
\]
First, notice that the smallest element  has an equation of type 1:
\[E_{1,4} = {\scriptstyle\bullet}~y^2  - \bold{y^3 R_{1,4}} - {\scriptstyle\bullet}~y^2 R_{2,4} - {\scriptstyle\bullet}~y^2 R_{3,4} = 0.\] 
This means it cannot have maximal degree among the $R_{k,l}$'s.

For the remaining part of this example we will focus on the type 4 equations. These are:
\begin{eqnarray}
c + c R_{1,4} +  c R_{1,5} +  \underline{y^2 R_{1,6}} +  {\scriptstyle\bullet}~y R_{1,7} +  {\scriptstyle\bullet}~y R_{1,8} +  {\scriptstyle\bullet}~y R_{1,9} - &&\phantom{~=0 }\nonumber\\ 
 \rule{0pt}{2ex} {}- \bold{y R_{2,7}} - c R_{3,7} - c R_{4,7} - c R_{5,7} - c R_{6,7}&= &0, \label{r27}\\
\rule{0pt}{3ex} c + c R_{2,4} +  c R_{2,5} +  \underline{y^2 R_{2,6}} +  {\scriptstyle\bullet}~y R_{2,7} +  {\scriptstyle\bullet}~y R_{2,8} +  {\scriptstyle\bullet}~y R_{2,9}- &&\phantom{~=0}\nonumber\\
\rule{0pt}{2ex}- {\scriptstyle\bullet}~y R_{1,7} - \bold{y R_{3,7} }- c R_{4,7} - c R_{5,7} - c R_{6,7}&= &0.\label{r37}
\end{eqnarray}
By the proof, we want to get inequalities on the degrees of $R_{2,7}$, respectively $R_{3,7}$ from the equations (\ref{r27})  $E_{2,7}=0$ and (\ref{r37}) $E_{3,7}=0$. But the degree of their coefficients is not maximal among the other coefficients. To correct this we will use the  equations $E_{1,5}=0$, $E_{1,6}=0$ and $E_{2,6}=0$:
\begin{eqnarray}
{\scriptstyle\bullet}~y^2  - \bold{y^3 R_{1,5}} - {\scriptstyle\bullet}~y^2 R_{2,5} - {\scriptstyle\bullet}~y^2 R_{3,5} &= &0,\label{r15}\\
{\scriptstyle\bullet}~y^2  - \bold{y^3 R_{1,6}} - {\scriptstyle\bullet}~y^2 R_{2,6} - {\scriptstyle\bullet}~y^2 R_{3,6} &= &0,\label{r16}\\
 + c R_{1,4} +  {\scriptstyle\bullet}~y R_{1,5} +  {\scriptstyle\bullet}~y R_{1,6} +  c R_{1,7} +  c R_{1,8} +  c R_{1,9}-&& \phantom{~=0}\nonumber\\
 - \bold{y R_{2,6}} -c R_{3,6} &= &0.\label{r26}
 \end{eqnarray}
We modify $E_{2,7}$ in the following way:
$E_{2,7}:= yE_{2,7} + E_{1,6}.$
We obtain:
\begin{eqnarray}
{\scriptstyle\bullet}~y^2+ {\scriptstyle\bullet}~y R_{1,4} +  {\scriptstyle\bullet}~y R_{1,5} +  {\scriptstyle\bullet}~y^2 R_{1,6} +  {\scriptstyle\bullet}~y^2 R_{1,7} +  {\scriptstyle\bullet}~y^2 R_{1,8} +   &&\phantom{~=0 }\nonumber\\ 
 \rule{0pt}{2ex}{\scriptstyle\bullet}~y^2 R_{1,9}  - \mathbf{y^2 R_{2,7}} - {\scriptstyle\bullet}~yR_{3,7} - {\scriptstyle\bullet}~y R_{4,7} - {\scriptstyle\bullet}~y R_{5,7} - {\scriptstyle\bullet}~y R_{6,7}- && \phantom{~=0 }\label{r271}\\ 
\rule{0pt}{2ex}-{\scriptstyle\bullet}~y^2 R_{2,6} - {\scriptstyle\bullet}~y^2 R_{3,6} &= &0.\nonumber 
\end{eqnarray}
So, if $\deg(R_{2,7})>0$ is maximal, it has to be equal to the degree of one of the following:
\[ R_{2,6}, R_{3,6} \textup{~or~} R_{1,j}, \textup{~with~} j\in \{6,7,8,9\}.  \]
It is easy to notice that all these are smaller than $R_{2,7}$ in the defined order.

The equation for $R_{3,7}$ will be modified three times, namely:
\begin{eqnarray*}
E_{3,7}&:= &E_{3,7}+yE_{2,6}.\\
E_{3,7}&:= &yE_{3,7}+E_{1,5}.\\
E_{3,7}&:= &E_{3,7}+E_{1,6}.
\end{eqnarray*}
Thus we obtain a new equation from which we deduce that, if $\deg(R_{3,7})>0$ is maximal, then it is equal to the degree of one of the following:
\[R_{1,i},R_{2,j}, R_{3,k},\]
where $i \in\{4,\ldots,9\}$, $j\in \{5,\ldots,9\}$ and $k \in \{5,6\}$. So all of them are smaller than $ R_{3,7}$.
It is easy to notice that if some of the $R_{i,j}$ would be 0, this would only reduce the number of  cases we have to consider.

\subsection{Example 3}
Now we will give an example of how the proof of the surjectivity of $\psi$ works. We will start with an ideal $I \subset R$ with $\dim(R/\ini(I)) = 0$ and construct the corresponding  matrix of $\AA{\ini(I)}$.
 
 Let $I$ be the ideal generated by the following polynomials:
\begin{eqnarray*}
 f_0&= &x^3 - x^2y - 2xy^2 + 2y^3 - 2x^2 + xy + y^2 - x + 2y - 2,\\
 f_1&=&x^2y^2 - 2y^4 - x^3 + x^2y - 2y^3 + x^2 - 3xy + 4y^2 + 4x - y,\\
 f_2&=&xy^3 - y^4 - 2x^2y + 6xy^2 - 5y^3 + x^2 - xy + 2y^2 - 3x + 4y - 2,\\
f_3&=&y^5 + x^2y^2 - 2xy^3 + 2y^4 + 3xy^2 + 2y^3 - x^2 - 2xy - y^2 - x - 11y + 6.
\end{eqnarray*}
Its DRL initial ideal is $I_0 = \ini(I) = (x^3,x^2y^2,xy^3,y^5).$ So these polynomials are already a DRL \GB~ for $I$. So  We have  $t=3$, $m_0 =  0$, $m_1 = 2 $, $m_2 =  3,$ $m_3 = 5$ and $d_1 = 2,$ $d_2 = 1$, $d_3 = 2$. 
Notice  that in the support of $f_1$ there is a monomial divisible by a power of $x$ higher than or equal to $t$: $x^3$. So we will set $f_1$ to be $f_1 + f_0$.

The next step is to  compute the S-polynomials:
 {\setlength\arraycolsep{2pt}\begin{eqnarray*}
 S_{1,0} &= &y^2f_0 - xf_1,\\
 S_{2,1}&=& y\phantom{^2}f_1 - x f_2,\\
 S_{3,2}&=&y^2f_2 - x f_3.
\end{eqnarray*}}
After performing the division algorithm we obtain:
\[ \begin{array}{rcrcrcrcr}
 S_{1,0} &= &(-1)f_0 &+ &y f_1  &+ &f_2 &+ &0f_3,\\
 S_{2,1}&=& (-2y + 1)f_0 &+& f_1 &+&(-y + 1)f_2 &+& f_3,\\
 S_{3,2}&=&(y^2 - 1)f_0 &+& 3f_1 &+& f_2 &+& (y + 1)f_3.
\end{array}\]
By Schreyer's theorem, these syzygies generate the syzygy module of $I$. So we have obtained the following Hilbert-Burch matrix:
\[\left(\begin{array}{rrr}
 y^2 - 1     &        -2y + 1   &          y^2 - 1\\
              -x + y     &          y + 1      &             3\\
                   1        &  -x - y + 1    &         y^2 + 1\\
                   0          &         1   &       -x + y + 1
 \end{array}\right).\]
 Notice that, as expected, it is a matrix of the form $X+A$. The matrix that bounds the degrees of the entries of the matrices in $\AA{I_0}$ is
 \[\left(\begin{array}{ccc}
 1 & 1 & 1\\
  1 & 0 & 0\\
  1 & 0 & 1\\
  0 & 0 & 1
 \end{array}\right),\textrm{~and ~} A = 
 \left(\begin{array}{rrr}
  - 1     &        -2y + 1   &          y^2 - 1\\
               + y     &            1      &             3\\
                   1        &  - y + 1    &           1\\
                   0          &         1   &       - y + 1
 \end{array}\right),\]
so  $A\notin \AA{I_0}$. We will need to do some reduction moves. We will start looking at the upper left $2 \times 1$ corner of $A$. There the bounds are respected. Now we will look at the up upper left $3 \times 2$ corner. We start looking at the last row of this block, from right to left. Then, if everything is fine there, we look at the last column from top to bottom. In this example, the first entry that we look at, ($a_{3,2}$) has degree higher than the bound. So we apply the  reduction move Red$_{3,2}$ to $X+A$:
\begin{itemize}
\item[-] Subtract from  row 3, row 2 multiplied by (-1) . 

\item[-] As you can see, in position $(3,1)$  there is an entry which contains $x$. So to cancel this $x$ we subtract from column 1:  column 2 multiplied by (1). We obtain: 
 \end{itemize}
\[{\scriptsize\left(
\begin{array}{rrr}
  y^2 - 1         &         -2y + 1          &        y^2 - 1\\
                   -x + y       &             y + 1          &              3\\
               -x + y + 1       &            -x + 2    &              y^2 + 4\\
                        0               &         1        &       -x + y + 1
  \end{array}
  \right)
  ,\mathrm{~then~}
  \left(
  \begin{array}{rrr}
  y^2 + 2y - 2             &     -2y + 1             &     y^2 - 1\\
                   -x - 1         &           y + 1          &              3\\
                    y - 1           &        -x + 2       &           y^2 + 4\\
                       -1              &          1           &    -x + y + 1
  \end{array}
  \right).
  }\]
Now we start over with  checking the matrix. This time we find an entry with degree higher than the bound in position $(1,3)$. We apply Red$_{1,3}$:
\begin{itemize}
\item[-] Subtract from column 3, column 1 multiplied by 1.
\item[-] Subtract from row 2, row 4 multiplied by 1. We obtain:  
\end{itemize}
\[{\scriptsize\left(\begin{array}{rrr}
 y^2 + 2y - 2            &      -2y + 1            &      -2y + 1\\
                   -x - 1       &             y + 1       &             x + 4\\
                    y - 1       &            -x + 2      &        y^2 - y + 5\\
                       -1          &              1          &     -x + y + 2
\end{array}\right),\mathrm{~then~}
\left(\begin{array}{rrr}
  y^2 + 2y - 2              &    -2y + 1          &        -2y + 1\\
                   -x - 2          &          y + 2        &            y + 6\\
                    y - 1          &         -x + 2       &       y^2 - y + 5\\
                       -1            &            1          &     -x + y + 2
\end{array}\right).}\]

We check again the matrix in the same order and find that the entry $(2,3)$ does not respect the upper bound. Notice that this entry was of lower degree when we started. So we apply now Red$_{2,3}$:
\begin{itemize}
\item[-] Subtract from column 3, column 2 multiplied by (-1).

\item[-] Subtract from row 3, row 4 multiplied by (-1). We obtain:  
\end{itemize}
\[{\scriptsize\left(\begin{array}{rrr}
 y^2 + 2y - 2           &       -2y + 1              &          0\\
                   -x - 2        &            y + 2         &               4\\
                    y - 1         &          -x + 2       &   y^2 + x - y + 3\\
                       -1           &             1           &    -x + y + 1
\end{array}\right),\mathrm{~then~}
\left(\begin{array}{rrr}
 y^2 + 2y - 2          &        -2y + 1          &              0\\
                   -x - 2     &               y + 2      &                  4\\
                    y - 2       &            -x + 3      &            y^2 + 4\\
                       -1        &                1      &         -x + y + 1
\end{array}\right).}\]

And now, after checking again, we find that this time the matrix respects all the upper bounds. So the matrix $A' \in \AA{I_0}$ that corresponds to the ideal $I$ is:
\[\left(\begin{array}{rrr}
  2y - 2          &        -2y + 1          &              0\\
                    - 2     &               2      &                  4\\
                    y - 2       &             3      &             4\\
                       -1        &                1      &          y + 1
\end{array}\right).\]
The generators of $I$ given by the signed minors of the Hilbert-Burch matrix have changed. They are now:
{\setlength\arraycolsep{2pt}\begin{eqnarray*}
f_0'&=&x^3 - x^2y - 2xy^2 + 2y^3 - 2x^2 + xy + y^2 - x + 2y - 2,\\
f_1'&=&x^2y^2 - xy^3 - y^4 + 2x^2y - 8xy^2 + 5y^3 - 2x^2 - xy + 3y^2 + 6x - 3y,\\
f_2'&=&xy^3 - y^4 - 2x^2y + 6xy^2 - 5y^3 + x^2 - xy + 2y^2 - 3x + 4y - 2,\\
f_3'&=&y^5 - 2xy^3 + 4y^4 + 5xy^2 + 2y^3 - 6y^2 - 4x - 12y + 8.
\end{eqnarray*}}

\section{Ideals in $K[x,y,z]$}

In this section we will  consider ideals of the polynomial ring in three variables. Given any monomial ideal $J_0$ of $K[x,y,z]$ and considering the affine variety of the homogeneous ideals that have $I_0$ as initial ideal for a certain term order $\tau$, we do not obtain in general an affine space (see \cite{iarro2} and \cite{concavalla} for examples). We will prove that if we take $J_0 = I_0K[x,y,z]$, with $I_0 \in K[x,y]$ a lex-segment ideal, and choose the degree reverse-lexicographic order induced by $x>y>z$, then $V_{hom}(J_0)$ is again an affine space. We also give a parametrization for this space, which comes from the parametrization of \I{I_0}.

First we will introduce some notation and recall some results that we will use.

\subsection{Notation and useful results}
We will denote by $S:= K[x,y,z]$ and, as before, $R = K[x,y]$. We present now some known  results on homogenization and dehomogenization. Most of them can be found in a more general form in \cite{krobb}.

Let $f \in R$ and $F \in S$ be two polynomials. We will write  $f = c_1t_1 + \ldots + c_st_s$, with $c_i \in K$ and $t_i$ monomials in $x$ and $y$. We denote $u_i := \deg(t_i)$ and set $\mu := \max\{u_i\}$.
\begin{definition}\label{defhom}
\begin{itemize}
\item[a)] The homogenization of $f$ in $S$ is the following polynomial $$f^{hom}:= \sum_{i=1}^s c_it_i z^{\mu - u_i}$$.\\
The dehomogenization of $F$ with respect to $z$ is
$F^{deh}:= F(x,y,1).$
\item[b)] Let $I \subset R$ and $J \subset S$ be two ideals. The homogenization of $I$ in $S$ is the ideal
\[I^{hom}:= (f^{hom} ~:~ f \in I )~ \subseteq~ S.\]
The dehomogenization of $J$ with respect to the variable $z$ is
\[J^{deh}:= (F^{deh} ~:~ F \in J )~ \subseteq ~R.\]
\end{itemize}
\end{definition}
\noindent Here are some remarks on the behavior of polynomials and ideals under the two operations defined above.
\begin{proposition}\label{homid}
Consider $f,g \in R$ and $F,G \in S$ and let $I \subset R$ and $J \subset S$ be two ideals. 
\begin{itemize}
\item[\textup{1.}] $(f^{hom})^{deh} = f$.
\item[\textup{2.}] If $s = \max \{ i ~:~ z^i ~\textup{divides}~ F\}$ then: $z^s(F^{deh})^{hom} = F$.
\item[\textup{3.}] $(I^{hom})^{deh} = I$.
\item[\textup{4.}] $J \subseteq (J^{deh})^{hom} = J :_{S} (z)^{\infty}$.
\item[\textup{5.}] If $I \neq R$ then $z$ is a non-zero divisor of $S/I^{hom}$.
\end{itemize}
\end{proposition}

On both $R$ and $S$ we will always consider  the degree reverse-lexicographic term order.  As this term order is degree compatible, from \cite{krobb}, Chapter 4.3 we can deduce the following. 
\begin{proposition} \label{inhom}
Let $f \in R$ a non-zero polynomial and  $I$ be an ideal of $R$. Let  $F \in S$ be a non-zero homogeneous polynomial and $J$ be a non-zero homogeneous ideal of $S$. Then
\begin{itemize}
\item[\textup{1.}] $\ini(f^{hom}) = \ini(f)$ and $\ini(F^{deh}) = (\ini(F))^{deh}$.
\item[\textup{2.}] If $\{f_1,\ldots, f_s\}$ is a \GB ~of $I$, then $\{f_1^{hom},\ldots,f_s^{hom}\}$ is a \GB~ of $I^{hom}$.
\item[\textup{3.}]If $\{F_1,\ldots, F_s\}$ is a homogeneous \GB ~of $J$, then $\{F_1^{deh},\ldots,F_s^{deh}\}$ is a \GB~ of $J^{deh}$.
\end{itemize}
\end{proposition}

Now we will define similar operations on matrices. In particular, the dehomogenization of a matrix $A$ with entries $S$ with respect to the variable $z$ will be just the dehomogenization of all its entries. We will denote this new matrix, with entries in $R$ by $A^{deh}$.

The homogenization of a matrix with entries in $R$ will not be defined this straight forward. We will define this only for the matrices that parametrize \I{I_0}.

Let $I_0 \subseteq R$ be a monomial ideal generated, as in the previous chapter, by $x^t,x^{t-1}y^{m_1},\ldots,y^{m_t}$. We recall from Section 2 its degree matrix, that is the $(t+1)\times t$ matrix $U(I_0)$ with entries:
\[u_{i,j}= m_j - m_{i-1} +i -j.\]
Now we can define the homogenization of a matrix. Notice that this will depend on the degree matrix associated to $I_0$.

\begin{definition}
Let $A \in \AA{I_0}$, with entries $a_{i,j}$. For every $i = 1,\ldots, t+1$ and $j = 1, \ldots, t$ we define:
\[ a_{i,j}^{\underline{hom}}:= z^{u_{i,j}- \deg(a_{i,j})}a_{i,j}^{hom},\] 
where $a_{i,j}^{hom}$ is the standard homogenization defined in \ref{defhom}. The homogenization of the matrix $A$ will be the matrix with entries $a_{i,j}^{\underline{hom}}$. We will denote this matrix by $A^{\underline{hom}}$. 
\end{definition}
\begin{remark} We could define the homogenization in the same way also for the matrix $X+A$. But as the entries of $X$ are either $0$ or of degree $u_{i,j}$ we would have
\[ (X+A)^{\underline{hom}} = X+(A^{\underline{hom}}). \]
\end{remark}
The matrices $A^{\underline{hom}}$ and $X + A^{\underline{hom}}$ are  homogeneous matrices in the sense of Definition 4.7.1. of \cite{krobb}. So their minors will be homogeneous polynomials in $S$. In particular, the ideal generated by the maximal minors of $X+A^{\underline{hom}}$ is a homogeneous ideal of $S$. 

For $i = 0, \ldots, t$  will denote by $f_i$  the determinant of the matrix obtained from $X + A$ by deleting the $(i+1)$th row  times $(-1)^{i+1}$, and by $F_i$ the determinant of the matrix obtained from $X + A^{\underline{hom}}$ by deleting the $(i+1)$th row times $(-1)^{i+1}$. It is easy to see that we have:
\[ F_i = (f_i)^{hom}. \]

We will end this section with a lemma that will turn out useful later.

\begin{lemma}\label{lemahomid}
Let $A \in \AA{I_0}$ be a matrix. With the above notations we have: 
\begin{itemize}
  \item[\textup{1.}]  $(I_t(X+A))^{hom} = I_t(X+A^{\underline{hom}}).$
  \item[\textup{2.}]  $I_t(X+A) = (I_t(X+A^{\underline{hom}}))^{deh}.$
\end{itemize}  
\end{lemma}
\begin{proof}
In the proof of Theorem \ref{main} we have seen that the set $\{f_0,\ldots,f_t\}$ forms a degree reverse lexicographic \GB~ of $I_t(X+A)$. So, by Proposition \ref{inhom} we have that the set $\{F_0,\ldots,F_t\}$ forms a degree reverse lexicographic \GB~ of $(I_t(X+A))^{hom}$. Thus the first part follows.
The second part is an immediate consequence of the third point of Proposition \ref{homid}.
\end{proof}

\subsection{Parametrization}

Using the parametrization given by Theorem \ref{main}, we will now parametrize the following variety. Let $J \subset S$ be  a  Cohen-Macaulay homogeneous ideal with Krull dimension $\dim(S/J) = 1$ and such that $z$ is not a zero divisor for $S/J$. This implies  that $J$ defines   a zero-dimensional subscheme of $\mathbb{P}^2\setminus\{z=0\}$. In the following remark the fact that we use the DRL order with $z$ the smallest variable is curcial.
\begin{remark}
The fact that $z$ is not a zero divisor for $S/J$ is equivalent   to $\ini(J)$ being generated by monomials that are not divisible by $z$.
\end{remark}
\begin{proof}
If there would exist a minimal generator of $\ini(J)$ divisible by $z$, given the fact that we use the degree reverse lexicographic term order,  we would find a homogeneous generator of $J$ that would be a multiple of $z$. 
Now suppose $z$ is a zero divisor and choose $f \in S \setminus J$, such that $zf \in J$ and $\ini(f)$ is minimal with this property. As $z\ini(f) \in \ini(J)$, which is generated by 
monomials in $x$ and $y$, we obtain $\ini(f) \in \ini(J)$. So there exists a polynomial $g \in J$ with $\ini(f) = \ini(g)$. As $f-g \notin J$, $z(f-g) \in J$ and $\ini(f-g) < \ini(f)$ we obtain a contradiction.
\end{proof}
Denote $\ini(J) = J_0$. The ideal  $J_0$ will be  of the form:
\[J_0 = I_0S, \quad \textup{with } I_0 \subset R,~ \textup{a monomial ideal.}\]
We will consider the ideals for which $I_0$ is just as in the hypothesis of Theorem \ref{main}. So we have that $\dim_K(R/I_0) < \infty$ and we will also require $I_0$ to be a lex-segment ideal. For this type of ideals we will parametrize the following affine variety:
\[
\V{J_0} =  \{ J \subset S ~:~J ~\textup{is a homogeneous ideal with}~ \ini(J)=J_0 =  I_0S\}.
\]
We will prove that this variety is parametrized also by \AA{I_0}. Recall that \AA{I_0} was the set of matrices with entries polynomials in $y$, that satisfy (\ref{A}). We define the following application:
\begin{displaymath}
\begin{array}{c}
 \overline{\psi}: \AA{I_0} \longrightarrow \V{J_0} \\
 \rule{0pt}{3ex}\overline{\psi}(A) = I_t(X+A^{\underline{hom}}), \quad \textup{for all} ~A \in \AA{I_0}.
\end{array}
\end{displaymath}

\begin{theorem}\label{main2} Let $J_0 = I_0S \subset S$ be a monomial ideal, where $I_0$ is a lex-segment ideal of $R$ such that $\dim_K(R/I_0) < \infty$. Then the application $\overline{\psi}: \AA{I_0} \longrightarrow \V{J_0}$ defined above is a bijection.
\end{theorem}

\begin{proof}
In order to prove the theorem we need to prove again three things:
\begin{enumerate}
\item 
The application $\overline{\psi}$ is well defined.
\item
The application $\overline{\psi}$ is injective.
\item
The application $\overline{\psi}$ is surjective.
\end{enumerate}

\emph{Proof of 1.} For every $A$ in $\AA{I_0}$, denote the ideal $I_t(X+A^{\underline{hom}})$ by $J_A$. We need to show that $J_A$ is homogeneous and has $\ini(J_A) = J_0$. Using the notation of the previous section, we have by definition that the polynomials $F_0,\ldots,F_t$ are homogeneous. We just need to show that they form a \GB~ and that their initial terms generate $J_0$.

We know from Theorem \ref{main} that $f_0,\ldots,f_t$ form a \GB~ of $I_0$. 
As we have seen that for $i= 0,\ldots,t$ we have $(f_i)^{hom}=F_i$, by applying Proposition \ref{inhom} we get that $F_0,\ldots,F_t$ form a \GB ~and that
$\ini(f_i) = \ini(F_i)$ for all $ i = 0,\ldots,t$. 

\emph{Proof of 2.} Let $A$ and $B$ be two matrices in $\AA{I_0}$. Suppose that $\overline{\psi}(A) = \overline{\psi}(B)$.  That is we have 
\[I_t(X+A^{\underline{hom}})=I_t(X+B^{\underline{hom}}).\]
By  Lemma \ref{lemahomid} we obtain that we also have
\[I_t(X+A) = I_t(X+B).\]
And by the injectivity of $\psi$ we get that $A = B$.\\

\emph{Proof of 3.} Let $J \in \V{J_0}$ be a homogeneous ideal. By Proposition \ref{inhom} we have that $J^{deh} \subset R$ is an ideal that has $\ini(J^{deh}) = I_0$. So by Theorem \ref{main} we know that 
\[J^{deh} = I_t(X+A), \quad \textup{for some }~ A \in \AA{I_0}.\]
We will show that $J =I_t(X+A^{\underline{hom}})$. By Lemma \ref{lemahomid} we have that 
\[I_t(X+A^{\underline{hom}}) = (I_t(X+A))^{hom} = (J^{deh})^{hom}.\]
 To complete the proof we just need to show that $J=(J^{deh})^{hom}$. By Proposition \ref{homid} this means we have to show that $J = J :_{S} (z)^{\infty}$. But this is equivalent to $z$ not being a zero divisor for $S/J$.
\end{proof}

\section{Betti strata}

We will now fix a Hilbert series $H$ and consider  all ideals $J$ defining zero-dimensional subschemes   of  $\mathbb{P}^2$ such that the Hilbert series of $S/J$ is $H$. As before, by such an ideal  we understand a homogeneous ideal $J \subset S$ such that $S/J$ is Cohen-Macaulay of dimension 1. In this case we have that the maximal ideal $\mathfrak{m} = (x,y,z)$ of $S$ is not an associated prime of $S/J$, which is equivalent to $J$ being a saturated ideal. So, the Hilbert series $H$ will be of the form:
\[ H(s) = \frac{h(s)}{1-s},\]
with $h(s)$ the Hilbert series of the zero-dimensional algebra $S/J + (\ell)$, where $\ell$ is a linear non-zero divisor of $S/J$.

In this section we will assume that the field $K$ is algebraically closed and that its characteristic is either 0 or  $p> \max\{j~:~h_j\neq0\}$. We need this assumption in order to be able to say that the generic initial ideal is a strongly stable ideal. We recall here that  a monomial ideal $I$ of $\k{n}$ is strongly stable if for any monomial $M \in I$ and every $1\le j < i\le n$ such that $x_i$ appears in $M$, we have $(x_i/x_j)M \in I$.

 Denote by 
\[\mathbb{G}(H) = \{ J \subset S ~:~J~ \textup{defines a 0-dimensional scheme and}~ H_{S/J} = H \}\] 
 the variety that parametrizes graded saturated ideals of $S$ such that the Hilbert series of $S/J$ is $H$.

The first restriction that we will use will be to consider ideals for which $z$ is not a zero divisor. Let $J$ be an ideal defining a zero-dimensional subscheme of $\mathbb{P}^2$. This also means that 
\[ J = q_1 \cap \ldots \cap q_s, \] 
where for all $i$ we have$\sqrt{q_i} = p_i$ and $p_i$ is the ideal of a point $P_i$ in $\mathbb{P}^2$.
The geometric equivalent of $z$ not being a zero divisor for $S/J$ is that none of  the points $P_1,\ldots, P_s$  belongs to the line of $\mathbb{P}^2$ given by $z = 0$.  This means that the set:
\[ \mathbb{G}^*(H) := \{ J \in \mathbb{G}(H) ~:~ z ~\textup{is a non-zero divisor for}~ S/J\}\]
is an open subset of $\mathbb{G}(H)$.

Due to the choice of the term order, the fact that $z$ is not a zero divisor for $S/J$ implies that $\ini(J) = IS$, where $I$ is an ideal of $R$. We also have that $H_{R/I} (s) = h(s)$. The same thing also holds for the degree reverse lexicographic generic initial ideal of $J$. So we have that: 
\[ \Gin(J) = I_0S, \quad \textup{where}~ I_0 \subset R.\]
Due to the assumption on the characteristic of $K$,  the generic initial ideal is strongly stable, so we also get that $I_0$ must be strongly stable. But in  $R = K[x,y]$ the only strongly stable ideal with that Hilbert series is $\Lex(h)$.
This means that the set:
\[ \mathbb{G}^*_{\Lex} (H) = \{ J \in \mathbb{G}^*(H) ~:~ \ini(J) = \Lex(h)S\} \]
is an open subset of $\mathbb{G}^*(H)$. \\

Notice  that, if   $J_0=\Lex(h)S$, then for all $J\in\V{J_0}$, as $J_0$ is generated only by monomials in $x$ and $y$, the maximal irrelevant ideal of $S$ is not an associated prime, thus   $V_{hom}(J_0)=\mathbb{G}^*_{\Lex} (H)$.  We  study  the Betti strata of this affine set. 

For a homogeneous ideal $J\subset S$ we will denote by $\b_{i,j}(J)$ the $(i,j)$th Betti number. In particular, $\b_{0,j}(J)$ is the number of minimal generators of $J$ of degree $j$. It is known that any two of the  sets $\{\b_{0,j}(J)\}_j$, $\{\b_{1,j}(J)\}_j$ and $\{\dim(J_j)\}_j$ determine the third. For the fixed Hilbert series $H(s) =  h(s)/(1-s)$ and for given integers $j$ and $u$ we define:   
\[\begin{array}{rcc}
 V(H,j,u) &= &\{ J \in  \mathbb{G}^*_{\Lex} (H) ~:~ \b_{0,j}(J) = u\},\\
 \rule{0pt}{3ex} V(H,j,\ge u) &=& \{ J \in  \mathbb{G}^*_{\Lex} (H) ~:~ \b_{0,j}(J) \ge u\}.
 \end{array}\]
For a vector $\b = (\b_1,\ldots,\b_j,\ldots)$ with integral entries we define:

\[\begin{array}{rccl}
 V(H,\b) &= &{\bigcap}_{j} &V(H,j,\b_j), \\
 \rule{0pt}{3ex} V(H,\ge \b) &= &\bigcap_{j} &V(H,j,\ge\b_j).
 \end{array}\]

 For the fixed Hilbert function $H$, we consider the lex-segment ideal $\Lex(h)$ and denote by $m_0,\ldots, m_t$ its associated sequence of integers from Section 2. 
  We have shown that $\mathbb{G}^*_{\Lex} (H)$ is parametrized by $\AA{\Lex(h)}$, which is an affine space $\mathbb{A}^N$.
So we know that to each ideal $J \in \mathbb{G}^*_{\Lex} (H)$  corresponds a unique matrix $A \in \AA{\Lex(h)}$. Starting from this matrix $A$ we can construct a Hilbert-Burch matrix, that is $X+A^{\underline{hom}}$. For simplicity we will denote
$M:= X+A^{\underline{hom}}.$
We have by the Hilbert-Burch theorem the following free resolution:
\begin{equation}\label{rez}
 0 \longrightarrow \bigoplus_{i=1}^{t}S(-q_i) \stackrel{M}{\longrightarrow} \bigoplus_{i=1}^{t+1} S(-p_i) \longrightarrow J \longrightarrow 0
 \end{equation}
where $p_i = t+1 -i + m_i$ for $i = 1,\ldots, t+1$   and $q_i = p_i +1$ for $i = 1,\ldots,t$. For every integer $j$ we  define the sets of indices:
\[w_j = \{i~:~ p_i = j\} \quad \textup{and} \quad v_j=\{i~:~ q_i = j\}.\]
For every integer $j$ denote by $M_j$ the submatrix of $M$ with row indices $w_j$ and column indices $v_j$. 
As we are considering matrices that are in \AA{\Lex(h)} we know that $0=m_0 < m_1 < \ldots < m_t$. So we also get $t+1=p_0 \le p_1 \le \ldots \le p_t$. This means that we can describe the matrices $M_j$ in terms of the $m_i$'s. They are the blocks of constants below the diagonal (see  Figure \ref{partA}). This means that  the entries of these matrices will be independent coordinates of  $\mathbb{A}^N$.

To compute the graded Betti numbers of $J$ we can tensor the resolution (\ref{rez}) with $K$ and look at the degree $j$ component. This will give us  the following complex of vector spaces, whose homology gives  the Betti numbers of $I$:
\[ K^{\sharp v_j} \stackrel{M_j}{\longrightarrow} K^{\sharp w_j} \longrightarrow 0,\]
where by $\sharp v_j$ (resp. $\sharp w_j$) we denote the cardinality of the set $v_j$ (resp. $w_j$).
Hence we have that:
\[\b_{0,j}(J) = \sharp w_j - \rank(M_j).\]
This means that $\b_{0,j}(J) \ge u$ is equivalent to
\[ \rank(M_j) \le \sharp w_j - u.\] 
Notice that, as we start from a matrix $A \in \AA{\Lex(h)}$, we have $\sharp w_j = \b_{0,j}(\Lex(h))$. That is the number of minimal generators of degree $j$ of $\Lex(h)$. We also have $\sharp v_j = \b_{1,j}(\Lex(h)) = \b_{0,j-1
}(\Lex(h))$.

So we obtain that $V(H,j,\ge u)$ is the determinantal variety given by the following  condition on the $\b_{0,j}(\Lex(h)) \times \b_{0,j-1}(\Lex(h))$ matrix:
\[ \rank(M_j) \le  \b_{0,j}(\Lex(h)) - u.\]

It is easy to notice, that for $i \neq j$ the sets of variables involved in $M_i$ and $M_j$ are disjoint. This means that the intersection $\bigcap_{j} V(H,j,\ge\b_j)$ is transversal.
We have proven the following:

\begin{proposition} Let $K$ be a field of characteristic  $p> \max\{j~:~h_j\neq0\}$ or of characteristic 0.
Each $V(H,j,\ge \b_j)$ is a determinantal variety and the variety $V(H,\ge\b)$  is the transversal intersection of the  $V(H,j,\ge\b_j)$'s.  The variety $V(H,j,\ge\b_j)$ is irreducible and it coincides with the closure of $V(H,j,\b_j)$, provided $V(H,j,\b_j)$ is not empty. 
\end{proposition}
 As a  corollary we have:
 \begin{corollary}Let $K$ be a field of characteristic  $p> \max\{j~:~h_j\neq0\}$ or of characteristic 0.
\begin{itemize}
\item[1.]The variety  $V(H,\ge\b)$ is irreducible.
\item[2.] The codimension of $V(H,\ge\b)$ in $\GG(H)$ is the sum of the codimensions of the $V(H,j,\ge\b_j)$'s in $\GG(H)$.
\end{itemize}
 \end{corollary}
The matrix $M_j$ is a matrix  of size $\b_{0,j}(\Lex(h)) \times \b_{1,j}(\Lex(h))$ whose entries are distinct variables. So,  whenever the variety $V(H,j,\ge u)$ is not empty, that is, whenever we have $\b_{0,j}(\Lex(h)) - \b_{1,j}(\Lex(h)) \le u \le \b_{0,j}(\Lex(h))$, its codimension is
\[ (\b_{0,j}(\Lex(h))-\b_{1,j}(\Lex(h))+ u)u.\]

If $J$ is a homogeneous ideal of the polynomial ring $S$ with Hilbert series $H$ and with $\b_{0,j}(J)=u$, then $\b_{0,j}(\Lex(h))-\b_{1,j}(\Lex(h))+ u = \b_{1,j}(J)$. This means that the formula for the codimension of $V(H,j,\ge u)$ can be written as 
 $\b_{1,j}(J)\b_{0,j}(J)$. We have thus obtained, by different methods  than the one indicated by the author in \cite[Remark 3.7]{iarro}, the generalization of the codimension formula  regarding the Betti strata for codimension two punctual schemes in $\mathbb{P}^2$, namely:

\begin{corollary}Let $K$ be a field of characteristic  $p> \max\{j~:~h_j\neq0\}$ or of characteristic 0.
Let $J \in \mathbb{G}(H)$ and set $\b = (\b_{0,j}(J))$. The variety $V(H,\ge\b)$ is irreducible, it is the closure of $V(H,\b)$ and it has codimension in $\GG(H)$
\[\sum_{j}\b_{1,j}(J)\b_{0,j}(J).\]
\end{corollary}

\end{document}